\documentclass{amsart}
\usepackage{graphicx} 
\usepackage{amsmath}
\usepackage{amsfonts}
\usepackage{amsthm}
\usepackage{amssymb}
\usepackage[all]{xy}
\usepackage{todonotes}
\usepackage{xcolor}

\DeclareMathOperator{\N}{\mathsf{N}}

\DeclareMathOperator{\Bun}{\mathsf{Bun}}

\DeclareMathOperator{\Nb}{\mathbf{N}}

\DeclareMathOperator{\Cb}{\mathbb{C}}
\DeclareMathOperator{\Fb}{\mathbb{F}}
\DeclareMathOperator{\Xc}{\mathcal{X}}

\DeclareMathOperator{\Oc}{\mathcal{O}}

\newcommand{\Gb}{\mathbb{G}}

\newcommand{\Dc}{\mathcal{D}}

\renewcommand{\lim}{\mathsf{lim}}

\DeclareMathOperator{\res}{res}

\newcommand\varto[1]{\mathrel{\hbox to #1pt{\rightarrowfill}}}

\DeclareMathOperator{\id}{id}

\DeclareMathOperator{\Frac}{Frac}

\DeclareMathOperator{\Ab}{\mathbb{A}}

\DeclareMathOperator{\Zb}{\mathbb{Z}}

\DeclareMathOperator{\G}{\mathbb{G}}

\DeclareMathOperator{\GL}{GL}

\DeclareMathOperator{\Y}{\mathcal{Y}}

\renewcommand{\Mc}{\mathcal{M}}
\DeclareMathOperator{\Sch}{\mathsf{Sch}}

\DeclareMathOperator{\Ss}{\mathcal{S}}

\DeclareMathOperator{\Spec}{\mathsf{Spec}}

\DeclareMathOperator{\Yc}{\mathcal{Y}}

\newcommand{\BN}{{\mathbb{N}}}

\DeclareMathOperator{\iso}{\mathsf{iso}}
\DeclareMathOperator{\AR}{\mathsf{AR}}
\DeclareMathOperator{\Fil}{\rm{Fil}}
\DeclareMathOperator{\Conn}{\mathsf{MIC}}
\DeclareMathOperator{\pConn}{\mathsf{p}-\mathsf{MIC}}
\DeclareMathOperator{\qnilp}{\rm qn}
\DeclareMathOperator{\FConn}{\mathsf{FMIC}}

\newcommand{\ga}[2]{\begin{gather}\label{#1}#2 \end{gather}}

\theoremstyle{definition}
\newtheorem{definition}{Definition}[section]

\newtheorem{cl}[definition]{Claim}
\newtheorem{example}[definition]{Example} 
\newtheorem{claim}[definition]{Claim}

\newtheorem{remark}[definition]{Remark}

\theoremstyle{plain}
\newtheorem{theorem}[definition]{Theorem}
\newtheorem{proposition}[definition]{Proposition}
\newtheorem{corollary}[definition]{Corollary}

\newtheorem{lemma}[definition]{Lemma}
\newtheorem{notation}[definition]{Notation}
\newtheorem{conjecture}[definition]{Conjecture}

\title{Cristallinity of rigid flat connections revisited}
\date{\today}

\begin{document}
\author{H\'el\`ene Esnault \and Michael Groechenig}
\address{Department of Mathematics, Freie Universit\"at Berlin}
\email{esnault@math.fu-berlin.de}
\address{Department of Mathematics, University of Toronto}
\email{michael.groechenig@utoronto.ca}
\begin{abstract}
We generalise a theorem on the existence of Frobenius  isocrystal  and Fontaine-Laffaille module  structures on rigid flat connections to the non-proper setting. The proof is based on a new strategy of a point-set topological flavour, which allows us to produce a purely $p$-adic statement and thereby to avoid the classical Simpson correspondence.
\end{abstract}
\thanks{ Michael Groechenig was supported by an NSERC discovery grant and an Alfred P. Sloan fellowship. 
}
\maketitle

\tableofcontents

\section{Introduction}

An irreducible local system $\mathcal{L}$ on a smooth projective complex variety $X$ is called \emph{rigid}, if it cannot be deformed (continuously or algebraically) to a non-isomorphic local system. This is tantamount to the corresponding point of the moduli space of local systems being isolated. 

The term local system can be understood either in the Betti or de Rham sense, i.e., as either being locally constant sheaf of complex vector spaces or a vector bundle with an integrable connection. Over the field of complex numbers those two theories produce equivalent categories by virtue of the Riemann-Hilbert correspondence.

\begin{conjecture}[Simpson]
A rigid local system $\mathcal{L}$ on $X$ is of \emph{geometric origin}. That is, there exists an open dense subset $U \subset X$ and a smooth projective family $\pi: Y \to U$ such that $\mathcal{L}$ is a direct summand of $R^{\bullet}\pi_*\Oc_Y$.
\end{conjecture}

In our previous work \cite{EG20} we proved that the restriction of a rigid flat connection on a $p$-adic model of  $X$
gives rise to an $F$-isocrystal, provided that $p$ is sufficiently large (with respect to a non-specified bound). This provides evidence for Simpson's motivicity conjecture, as for $p$ large such a Frobenius structure always exists 
on local systems of geometric origin.

In fact, reductions of rigid flat connections are endowed with an even stronger property: a Fontaine-Lafaille module structure. The recent proof of the Andr\'e-Oort conjecture by Pila--Shankar--Tsimerman (see \cite{PST}) relies on this in the context of local systems on Shimura varieties  of real rank $\ge 2$. The latter are often non-proper, which motivates the quest for a generalisation of our previous result to the non-proper setting. In the appendix to \cite{PST} we presented an argument, which was tailor-made for the case of Shimura varieties  of real rank $\ge 2$ and relied on a strong cohomological rigidity assumption known to hold in this case. In the present article, we settle this question in general  and 
 further present a new proof.

We 
now describe our main result. Let $k$ be a finite field of  odd characteristic. We fix a smooth proper scheme $\bar{X}/W(k)$ containing a strict relative normal crossings divisor $D/W$. The open complement $\bar{X} \setminus D$ 
is denoted by $X$.

 A $W$-linear flat connection $(E,\nabla)$ on $X$ is said to have quasi-unipotent mondromies at infinity if it extends to a logarithmic connection on $\bar{X}$, and  if the eigenvalues  $\lambda_{ij}$ of the residues of $\nabla$ along the components $D_i$ of $D$, which lie in $\bar {\mathbb Z}_p$, also lie in $\mathbb Q$ and have denominators prime to $p$. We refer to the type of $(E,\nabla)$ as the tuple consisting of the denominators in $\mathbb Q$ of the $\{\lambda_{ij}\}$.
  In the main body of the text we will use an equivalent characterisation involving root stacks.

The de Rham local system $(E,\nabla)$ is called \emph{rigid} over $K$, if it cannot be deformed (relatively to $K$) to a non-isomorphic de Rham local system, provided that the deformation fixes the boundary behaviour.
We will also assume stability, which, in characteristic zero, is equivalent to geometric irreducibility. 
\begin{theorem}[Main theorem] \label{thm:main}
Under the assumptions above, every local system on $X/W$ with fixed quasi-unipotent monodromies at infinity, which is rigid over $K$ gives rise to an overconvergent $F$-isocrystal.

Furthermore, if   $p> r$  and every rigid local system of the same type and rank admits a Griffiths-transverse filtration defined over $W$ which is locally split such that the associated Higgs bundle is slope  stable, then $(E,\nabla)$ can be endowed with the structure of a torsion-free Fontaine-Laffaille module.
\end{theorem}

We emphasise that even in the case of a projective variety $X$ (i.e., $D=\emptyset$) this result is actually stronger than the one of \cite{EG20}. It is indeed a purely local result over $W$ and does not refer to the objects coming from complex geometry. It notably yields a more explicit understanding of the mod $p$ reductions for which an $F$-isocrystal structure exists.


Furthermore, the proof given here follows a different strategy and is actually more elementary and self-contained. In particular, we avoid Simpson's \emph{non-abelian Hodge theory}.

 Lan--Sheng--Zuo's Higgs de Rham flows are crucial to our approach, nonetheless we present a slightly different approach to their construction, which hinges on Xu's lifting of the inverse Cartier operator \cite{Xu}. The bulk of this paper is devoted to developing a logarithmic analogue of this machinery.

The overall strategy we pursue could be classified as being point-set topological. We use that for a stack $\Xc$ over the ring of Witt vectors $W$, there is a natural topology on the set of isomorphism classes $\Xc(W)^{\iso}$. For the stack of de Rham local systems $\Mc_{dR}(W)$ we show that for $p\ge 3$ the Frobenius pullback map
$$F^*\colon \Mc_{dR}(W)^{\iso} \to \Mc_{dR}(W)^{\iso}$$
is well-defined, open and continuous. With this topological property at hand, we infer our main result using finiteness of rigid local systems of a fixed rank and bounded order of monodromies at infinity.

Our original approach was centred around a nilpotency result for the $p$-curvature of a rigid flat connection for $p$ large. In our new strategy, $p$-curvature no longer plays a central role. However, since every $W$-linear flat connection $(E,\nabla)$ with $(F^f)^*(E,\nabla) \simeq (E,\nabla)$ for some $f > 0$ must have nilpotent $p$-curvature, we immediately see that our main result implies the following corollary.

\begin{corollary}
Under the same assumptions as in Theorem \ref{thm:main}, the $p$-curvature of the special fibre of $(E,\nabla)$ is nilpotent.    
\end{corollary}

\subsubsection*{Leitfaden}

Section 2 is devoted to the topological preliminaries concerning the natural topology on $W$-points of stacks. We heavily rely on \v Cesnavi\v cius's article \cite{Cesnavicius}.

In the third section we extend the theory of Frobenius pullbacks and Higgs de Rham flows to the logarithmic setting. 

Our presentation is inspired by the theory of Lan--Sheng--Zuo (see \cite{LSZ}), but differs in various aspects, in particular, the use of a logarithmic analogue of Xu's inverse Cartier operator (see \cite{Xu}), which we also treat in this section (solely in the level of generality required by our arguments).

The fourth section contains the new argument. In particular, we prove that Frobenius pullback of $W$-linear flat connections is continuous and open. As an immediate consequence of the latter, we obtain that the Frobenius permutes isolated points. Since they are finite in number, they give rise to finite orbits and thus $F$-isocrystals.

\subsubsection*{Acknowledgements}
It is a great pleasure to acknowledge the influence and help of Johan de Jong. For proper varieties, the first purely $p$-adic proof of the existence of an $F$-isocrystal structure was developed by the first author and de Jong during a research visit at Columbia University. It is based on a characterisation of rigidity in terms of cup products. Details can be found in the lecture notes \cite[Chapter~8]{LN}. Furthermore, we are grateful to Johan de Jong for having drawn our attention to K\c estutis \v Cesnavi\v cius's article \cite{Cesnavicius}, which supplies 
the necessary topological foundations to our argument. We thank Peter Scholze for alerting us of an erroneous lemma in the first version of this article, Marco D'Addezio for an insightful remark and Will Sawin for pointing out an unnecessary hypothesis in one of our key lemmas.

\section{Topological preliminaries}
In this section we denote by $R$ a local topological ring, that is, a local commutative ring $R$ with a topology on the set $R$ for which the addition and the multiplication are continuous,
satisfying the axioms:
\begin{enumerate}
    \item[($\alpha$)] The  subset  $R^{\times} \subset R$ consisting of the units is open.
    \item[($\beta$)] The inverse map $R^{\times} \to R^{\times}$ is continuous  for the induced topology.
\end{enumerate}

For $R$-schemes, the topology on $X(R)$ was defined in Conrad's \cite{Conrad}. The case of stacks (and more generally, presheaves) is due to \cite{Cesnavicius}, building up on prior work of Moret-Bailly \cite{Moret} for topological fields.

In this article, we will only apply this construction to $W = W(k)$, the ring of Witt vector of a finite field $k=\Fb_q$, and to the topological field $K=\Frac W$.

\subsection{The case of a scheme}

We fix an  $R$-scheme $X$ with the additional quality that it is locally of finite presentation over the base ring $R$. 

\begin{proposition}
There exists a unique topology $\mathcal{T}$ on $X(R)$, called the \emph{strong topology}, with the following properties:
\begin{enumerate}
    \item For every Zariski-open $U\subset X$ the induced subset $$U(R)\subset X(R)$$ is open in $\mathcal{T}$.
    \item Let $U\subset X$ be Zariski-open and affine, and $i\colon U\hookrightarrow \mathbb{A}_R^N$ be a closed immersion, then the induced map 
    $$U(R)\hookrightarrow R^N$$
    is an embedding, that is, a homeomorphism onto its image.
\end{enumerate}
\end{proposition}
\begin{proof}
See \cite[Section 2.2]{Cesnavicius}.
\end{proof}

The following lemma can also be extracted from \cite[Section 2.2]{Cesnavicius}.
\begin{lemma}
The strong topology  is functorial and preserves finite products, i.e., it yields a functor
$\Sch/R \to \mathsf{Top}$
preserving finite products. Furthermore, a continuous ring homomorphism $R_1 \to R_2$ yields a continuous map $X(R_1) \to X_{R_2}(R_2)$ for every $R_1$-scheme $X$.
\end{lemma}

For the topological ring $R=W$  we will also refer to the strong topology as the \emph{$p$-adic topology}. In this case, more can be said. First of all, this functor sends $W$-schemes of finite type to compact Hausdorff spaces. 

\begin{lemma}\label{lemma:Hdff}
Let $X/W$ be an  algebraic $W$-space of finite type, then $X(W)$ is a compact Hausdorff space.    
\end{lemma}
\begin{proof}
This follows from \cite[Propositions 2.9(d) \& 2.17(8)]{Cesnavicius}. Since $W$ is Noetherian, every algebraic $W$-space of finite type satisfies the condition that the diagonal $\Delta_{X/W}$ is quasi-compact and separated.
\end{proof}

In fact, those compacts are Stone spaces (i.e., pro-finite sets), as shown by the following lemma.

\begin{lemma}\label{lemma:pro-finite}
Let $X/W$ be a an algebraic space of finite type over $W$, then the canonical map 
$$X(W) \to \varprojlim_{i \in \Nb} X(W_n)$$
is a homeomorphism, where the left-hand side is endowed with the $p$-adic topology and the right-hand side is the pro-finite 
topological space defined
by the inverse limit {of the finite sets}  in topological spaces.
\end{lemma}
\begin{proof}
The continuous ring homomorphism $W \to W_n$ induces a continuous map $X(W)\to X(W_n)$ for every $n\in \Nb$. Using the universal property of the inverse limit topology, we conclude that the map of the assertion above is continuous. By Lemma \ref{lemma:Hdff} it is a continuous bijection between compact Hausdorff spaces and thus a homeomorphism.
\end{proof}

\begin{lemma}\label{lemma:quotient-map}
Let $X/W$ be an algebraic $W$-space of finite type and $P\to X$ 
be a $G$-torsor  with respect to a smooth affine group scheme $G \to W$ with connected special fibre. Then, the $G(W)$-equivariant continuous map 
$$P(W) \to X(W)$$
is a quotient map, i.e., the continuous map
$$P(W)/G(W) \to X(W)$$
where the left-hand side is endowed with the quotient topology is a homeomorphism.
\end{lemma}
\begin{proof}
Functoriality of the $p$-adic topology implies right away that there is a continuous and $G(W)$-equivariant map $P(W) \to X(W)$. The universal property of the quotient topology yields a continuous map $P(W)/G(W) \to X(W)$.

Lang's Lemma yields the vanishing  $H^1_{\text{\'et}}(\Spec W, G) = 0$, and thus bijectivity of the quotient map. By Lemma \ref{lemma:Hdff},  the spaces $P(W)$ and $X(W)$ are compact Hausdorff spaces.
So the map $P(W)/G(W) \to X(W)$ is a continuous bijection between compact Hausdorff spaces, which implies
 that it is a homeomorphism.
\end{proof}

\subsection{The case of a stack} \label{sec:stack}

This subsection is devoted to assigning a topological space to an $R$-stack. First, we introduce a new piece of notation.

\begin{notation}
 Let $\Xc$ be a stack and $S$ be a scheme. We denote by $\Xc(S)$ the groupoid of $S$-points of $\Xc$ and by $\Xc(S)^{\iso}$ the set of isomorphism classes of $S$-points.    
\end{notation}

The following definition appeared in \cite[Section 2.4]{Cesnavicius}, and is based on \cite[Section 2]{Moret}, where the case of fields was covered.

\begin{definition}
A subset $U \subset \Xc(R)^{\iso}$ is called open if for every $R$-scheme $X$ locally of finite type over $R$ and for every morphism $g\colon X \to \Xc$, 
$g^{-1}(U)\subset X(R)$ is open with respect to the strong topology on $X(R)$.     
\end{definition}

For a large class of stacks, for example for  linear algebraic stacks which are defined below, we may render the definition of this topology more concrete.

\begin{definition}\label{defi:linear-quotient}
An $R$-stack $\Xc$ is called a \emph{linear quotient stack} if there exists a smooth and linear group $R$-scheme $G$ and an fppf $G$-torsor $Y \to \Xc$ such that $Y$ is an algebraic space. 
   We refer to such a torsor $Y\to \Xc$ as a linear atlas of $\Xc$.
\end{definition}

We emphasise that a linear group scheme $G$ is a closed subscheme of the general linear group $\GL_{N,R}$, and hence is of finite type.

Since $Y$ is a $G$-torsor, there is a canonical equivalence $\Xc \simeq [Y/G]$  where 
$[Y/G]$ is the quotient stack. In particular $\Xc$ is an algebraic stack.
Furthermore, we will frequently use the well-known observation that an embedding $G \hookrightarrow \GL_{N,R}$ allows one to replace the quotient $[Y/G]$ by 
$$\Xc \cong \big[Q_{Y;G}/\GL_{N,R}\big],$$
where $Q_{Y;G}$ denotes the quotient (as defined in the theory of algebraic spaces)
$G\backslash (Y\times_R \GL_{N,R})$ with respect to the anti-diagonal action. 

\begin{lemma}
 The quotient space $Q_{Y;G}$ defined above is an algebraic space.   
\end{lemma}
\begin{proof}
For the sake of simplifying the notation we  denote $\GL_{N,R}$ by $H$. There is a commutative diagram of $R$-stacks with cartesian squares
\[
\xymatrix{
Y \ar[r] \ar[d] & \Spec R \ar[d] & & \\
Q_{Y;G} \ar[r] \ar[d] & H / G \ar[r] \ar[d] & \Spec R \ar[d]  \\
\Xc \ar[r] & BG \ar[r] & BH
}
\]
 The morphism $\Spec R \to H/G$ is induced by the unit element of the $R$-scheme $H$. The morphism $H/G \to BG = [\Spec R / G]$ is induced by the $G$-invariant structure morphism $H \to \Spec R$.
The morphism $\Xc = [Y/G] \to BG=[\Spec R/G]$ is induced by the $G$-invariant morphism $Y \to \Spec R$, and thus is a representable morphism of algebraic stacks. In particular, we obtain that the fibre product $Q_{Y;G}$ is an algebraic space.
\end{proof}

This construction allows us to assume without loss of generality that the affine group $R$-scheme $G$ 
 is the general linear group.

\begin{proposition}\label{prop:quotient-stack}
Let $\Xc$ be a linear quotient $W$-stack which is locally of finite presentation. For every $G$-torsor $\pi\colon Y \to \Xc$ where $Y$ is an algebraic space 
 and $G$ is a smooth linear group $W$-scheme with connected special fibre, the map
$$Y(W) \to \Xc(W)^{\iso}$$
{on $W$-points}
is a topological quotient map. That is, there is a homeomorphism $$\Xc(W)^{\iso} \cong Y(W)/G(W).$$
\end{proposition}
\begin{proof}
We note that $Y$ an algebraic space and thus 
$Y(W)^{\rm iso}=Y(W)$.
Since $G/W$ is smooth and has connected special fibre, every $k$-rational point $\bar{x}$ of $\Xc$ lifts to $Y$. Indeed, the fibre $\pi^{-1}(\bar{x})$ is a $G_{k}$-torsor, and thus trivial by 
Lang's Lemma.

As $G$ is smooth, 
Hensel's lemma implies
that every $W$-point $x \in \Xc(W)$ lifts to a $W$-point of $Y$. This shows that the map $\pi\colon Y(W) \to \Xc(W)$ is surjective.  

Let $\tau\colon Z \to \Xc$ be a morphism, where $Z$ is a $W$-scheme locally of finite type.
We denote by $Q$ the base change $Y \times_{\Xc} Z$, which is itself an algebraic space by virtue of representability of the map $\pi$.
Furthermore, the base change $\sigma\colon Q \to Z$ is a $G$-torsor.

We now have the following commutative diagram in the category of topological spaces:
\[
\xymatrix{
Q(W) \ar[r]^{\sigma} \ar[d]_{\rho} & Z(W) \ar[d]^{\tau} \\
Y(W) \ar[r]^{\pi} & \Xc(W)^{\iso}.
}
\]
Let us assume that $U \subset \Xc(W)^{\iso}$ has the property that $\pi^{-1}(U)$ is open. Then, by continuity of $\rho$, the pre-image $\rho^{-1}(\pi^{-1}(U))= \sigma^{-1}(\tau^{-1}(U))$ is also open. Since by Lemma~\ref{lemma:quotient-map}. 
$\sigma$ is a quotient map,
 $\tau^{-1}(U) \subset Z(W)$ is open. This argument applies to all possible choices of a locally of finite type $R$-scheme $Z$ and a morphism $\tau$, and thus shows that $U \subset \Xc(W)^{\iso}$ is open in the strong topology. This shows that $\pi$ is a quotient map.
\end{proof}

\begin{lemma}\label{lemma:Stone-space2}
Let $\Xc$ be a linear quotient stack of finite type over $W$. Then, $\Xc(W)^{\iso}$ is a Stone space, i.e., a pro-finite set, or equivalently, a compact totally disconnected Hausdorff space.    
\end{lemma}
\begin{proof}
Let $P \to \Xc$ be a $G$-torsor 
 where $P$ is an algebraic $W$-space of finite type, 
providing us with a linear atlas for $\Xc$. By Proposition \ref{prop:quotient-stack}, the $p$-adic topological space $\Xc(W)^{\iso}$ is homeomorphic to the quotient space $P(W)/G(W)$. Since $\Xc$ and $G$ are of finite type over $W$, so is $P$, and we conclude that $P(W)$ is a compact Hausdorff space, as guaranteed by \cite[Propositions 2.9(d) \& 2.17(2)]{Cesnavicius}. These properties are directly inherited to the quotient $P(W)/G(W)$ 
as the group $G(W)$ is compact.

It remains to check that this quotient of $P(W)$ is totally disconnected. Let $x$ and $y$ be two distinct points in $\Xc(W)^{\iso}$. It suffices to construct disjoint open neighbourhoods $U$ and $V$ of $x$ respectively $y$ such that $\Xc(W)^{\iso}=U\cup V$.

Let $\tilde{x}$ be a lift of $x$ to $P(W)$ and likewise for $\tilde{y}$. Since $P(W)$ is a pro-finite set (see Lemma \ref{lemma:pro-finite}) we conclude that there exists a compact-open neighbourhood $\tilde{U}$ of $\tilde{x}$ which does not intersect the closed orbit $G(W)\cdot{}\tilde{y}$. 

We infer that $G(W)\cdot{} \tilde{U}$ is a $G(W)$-invariant compact-open neighbourhood of $G(W) \cdot{} \tilde{x}$ which does not intersect $G(W)\cdot{} \tilde{y}$. By definition of the quotient topology, it thus descends to a compact-open neighbourhood $U$ of $x$. By compactness, its complement $V$ is open, and by construction it is a neighbourhood of $y$.
\end{proof}

There is a concrete realisation of $\Xc(W)^{\iso}$ as a pro-finite set, given by the inverse limit
\begin{equation}\label{eqn:inverse-limit}
    \beta\colon \Xc(W)^{\iso} \xrightarrow{\cong} \varprojlim_{n\in \Nb} \Xc(W_n)^{\iso}.
\end{equation}
Note that the spaces $\Xc(W_n)^{\iso}$ are discrete and finite, and that the natural maps
  $\Xc(W)^{\iso} \to \Xc(W_n)^{\iso}$ are continuous. The verification that this map is a homeomorphism is the content of the following two assertions. We denote by $$\pi_{n}: 
   \varprojlim_{n'\in \Nb} \Xc(W_{n'})^{\iso}
  \to \Xc(W_n)^{\iso}$$ the natural map.

\begin{lemma}\label{lemma:basis}
 Let $\Xc$ be a linear quotient stack of finite type over $W$. There is a basis for the $p$-adic topology on $\Xc(W)^{\iso}$ given by the collection of all pre-images 
 $$(\pi_n\circ \beta)^{-1}(M)$$
 where $M\subset \Xc(W_n)^{\iso}$ is an arbitrary subset.
\end{lemma}

\begin{proof}
The collection of subsets $\{\pi_n^{-1}(M) | n \in \Nb\text{ and }M \subset \Xc(W_n)^{\iso}\}$ is a basis for the inverse limit topology. The assertion above is therefore equivalent to 
 $$\{(\pi_n\circ \beta)^{-1}(M) | n \in \Nb\text{ and }M \subset \Xc(W_n)^{\iso}\}$$
 being a basis for the $p$-adic topology on $\Xc(W)^{\iso}$.

  We fix a linear atlas $f\colon P \to \Xc$. By functoriality  of the $p$-adic topology, there is a commutative diagram 
 \[
 \xymatrix{
P(W) \ar[r]^{\tau_n} \ar[d]_{f(W)} & P(W_n) \ar[d]^{f(W_n)} \\
\Xc(W)^{\iso} \ar[r]^{\pi_n\circ \beta} & \Xc(W_n)^{\iso}
 }
 \]
 where the top row
 $\tau_n: P(W)\to \varprojlim_{n\in \mathbb N} P(W_n)\to P(W_n)$
  is $G(W)$-equivariant with respect to the natural action. On the  discrete space $P(W_n)$, the group $G(W)$ acts through the finite quotient $G(W_n)$. The vertical arrows are invariant with respect to the $G(W)$-action. 

 A subset of the shape $(\pi_n\circ \beta)^{-1}(M)$ is equal to the quotient of the open subset $\tau_n^{-1}(f(W_n)^{-1}(M))$ by $G(W)$. Therefore, it is open in the $p$-adic topology.

This shows that $U = \bigcup_{i \in I} \beta^{-1}\pi_{n_i}^{-1}(M_i)$, which concludes the proof that this collection of sets is a basis.
\end{proof}

\begin{corollary}
 Let $\Xc$ be a linear quotient stack of finite type over $W$. The map of equation \eqref{eqn:inverse-limit} is a homeomorphism. In particular, a subset $U\subset \Xc(W)^{\iso}$ is open if and only if it is the pre-image $\beta^{-1}(V)$ of an open subset $V$ of the inverse limit.
\end{corollary}
\begin{proof}
The second assertion follows directly from Lemma \ref{lemma:basis}. Indeed, the topology generated by this basis agrees with the collection of pre-images of open subsets in the inverse limit.
We claim that the map $\beta$ must be injective since $\Xc(W)^{\iso}$ is Hausdorff (see Lemma \ref{lemma:Stone-space2}). Indeed, if two distinct points of $\Xc(W)$ were sent to the same point of the inverse limit, then it would be impossible to separate them by the open subsets provided by Lemma \ref{lemma:basis}. Surjectivity of $\beta$ follows from  Grothendieck's Existence Theorem. We conclude that $\beta$ is a homeomorphism.
\end{proof}

\begin{corollary}\label{cor:continuous}
 Let $T$ be an arbitrary topological space. A set-theoretic map $T\to \Xc(W)^{\iso}$ is continuous if and only if the composition $T \to \varprojlim \Xc(W_n)^{\iso}$ is continuous.    
\end{corollary}

\section{The Higgs-de Rham flow revisited}

The purpose of this section is twofold: to recast the Higgs-de Rham flow defined in \cite{LSZ} by Lan--Sheng--Zuo using Xu's inverse Cartier operator \cite{Xu}, and to simultaneously extend it to a more general setting where the connection $\nabla$ is allowed to have \emph{logarithmic poles} with \emph{nilpotent} residues.

\subsection{Frobenius pullbacks for $W$-families of flat connections}

In this subsection of expository nature we recall the classical construction of Frobenius pullbacks for relative flat connections $(E,\nabla)$ defined on a smooth scheme $X/W$. 
Henceforth we assume $p\geq 3$, we do not impose any nilpotence conditions on the $p$-curvature of the special fibre $(E,\nabla)_k$, as is usually done in the literature. For simplicity we assume that $E$ is coherent and locally free.
 The Frobenius pullback exists on this level since we work over the non-ramified base ring $W$. \footnote{See \cite[Chapter 8]{LN}. We thank A. J. de Jong for this crucial remark and M. Kisin for the reference to \cite{Ber00}.}

Every $x \in X_k$ possesses an affine open neighbourhood $U\subset X$, equipped with a locally defined formal Frobenius lift 
$$\widehat{F} \colon \widehat{U} \to \widehat{U},$$
where $\widehat U$ is an affine flat formal lift of $U$ and 
$\tilde{F}|_{U_k}\colon U_k \to U_k$ 
is equal to the usual absolute Frobenius morphism sending a section $f \in \Oc_{X_k}(U_k)$ to $f^p$.

Such a chosen local lift allows one to consider the pullback connection $\tilde{F}(E,\nabla)|_{\widehat{U}}$. By \cite[Proposition 2.2.5 a)]{Ber00}, there exists a natural descent datum provided by the Taylor formula for such a collection of locally defined Frobenius pullbacks which we discuss now. 

In order to define the descent datum, we must assume that $U$ is equipped with \'etale coordinates, that is, an \'etale morphism $U \to \Ab^d_{W}$. This can always be achieved on a sufficiently small open subset \cite[tag 054L]{StPr}. We then denote by $(x_1,\ldots, x_d)$ the coordinate functions taking values in $\Ab^1_W$. We denote by $\partial_i$ the derivation dual to $dx_i$. 
For the self-map $\widehat{F}\colon \widehat{U}\to \widehat{U}$ we write $\widehat{F}_j$ to denote the composition $x_j \circ \widehat{F}$. We use the notation $\widehat{\mathcal O}=\varprojlim_n \mathcal{O}_{U_n}$ for the sheaf of formal functions on $U$. 
\begin{definition}
Let $(\widehat{E},\widehat{\nabla})$ be a fixed formal flat connection on $\widehat{U}$ and $$\widehat{F}, \widehat{G}\colon \widehat{U} \to \widehat{U}$$ be local Frobenius lifts. We define an $\widehat{\Oc}$-linear map
$$\psi_{\widehat{G}\widehat{F}}\colon \widehat{G}^*E \to \widehat{F}^*E$$
by the formula (see \cite[Proposition 2.2.5 a)]{Ber00})
\begin{equation}\label{eqn:psi}\psi(s \otimes 1) = \sum_{\underline{i}\in \Nb^{d}} \partial^{\underline{i}}(s) \otimes \frac{\prod_{j=1}^d(\widehat{F}_{j}-\widehat{G}_{j})^{i_j}}{\underline{i}!},\end{equation}
where $s \in \Gamma(U,\widehat{G}^{-1}\widehat{E})$.
\end{definition}

In the above, the notation $\partial^{\underline{i}}$ is used as shorthand for $\partial_1^{i_1}\circ \cdots \circ \partial_d^{i_d}$. Similarly, $\underline{i}!$ refers to the product $\prod_{j=1}^d (i_j!)$.

Note that the infinite series above converges whenever $p > 2$, since the difference $(\widehat{F}_i - \widehat{G}_i)$ belongs to $p\widehat{\Oc}_U$. This follows from convergence of the $p$-adic exponential series on $pW \subset W$, see \cite[Chapter 12]{Cassels}. Here, it becomes crucial that we work over $W$, as the series would diverge in a ramified field extension.

The proof of the following two identities is classical and is left as an exercise to the reader.

\begin{lemma}\label{lemma:ab}
The map $\psi_{\widehat{G}\widehat{F}}$ satisfies the relations 
\begin{enumerate}
    \item[(a)] $\widehat{F}^*(\widehat{\nabla})\circ \psi_{\widehat{G}\widehat{F}} = \psi_{\widehat{G} \widehat{F}} \circ \widehat{G}^*(\widehat{\nabla})$;

    \item[(b)] $\psi_{\widehat{H}\widehat{G}}\circ \psi_{\widehat{G}\widehat{F}} = \psi_{\widehat{H}\widehat{F}}$, where $\widehat{H}$ is another local Frobenius lift.
\end{enumerate}
\end{lemma}
    
Identity (a) and (b) above imply that the maps $\psi_{\widehat{G}\widehat{F}}$ are a collection of isomorphisms $\widehat{F}^*(\widehat{E},\widehat{\nabla}) \simeq \tilde{G}^*(\widehat{E},\widehat{\nabla})$ satisfying the cocycle identity. By covering $X$ with sufficiently small open subsets $U$ endowed with \'etale coordinates and local Frobenius lift, we can glue the locally defined pullbacks. 
The resulting flat connection relative to $W$ will be denoted by $F^*(\widehat E, \widehat \nabla)$, despite the fact that there rarely is a globally defined Frobenius lift. This construction is of functorial nature, and the resulting functor is denoted by
$$F^*\colon \Conn(\widehat{X}/W) \to \Conn(\widehat{X}/W).$$
{ Here, $\Conn(\widehat{X}/W) $ denotes the category of vector bundles on $X$ with an integrable connection relative to $W$.

It will be crucial for us that there exists a logarithmic generalisation of this construction.

Let us consider a strict relative  normal crossings divisor $D \subset X$.
 Let us denote by $\Conn(X/W,D/W)$ the category of pairs $(E,\nabla)$ where $E$ is a vector bundle  on $X$ and $\nabla\colon E \to E \otimes \Omega_{X/W}^1(\log D)$ is an integrable 
 logarithmic connection relative to $W$.  Similary, let us denote by $\Conn(\widehat X/W, \widehat D/W)$ the category of pairs $(\widehat E, \widehat\nabla)$ where $\widehat E$ is a formal vector bundle on $\widehat X$ and $\widehat \nabla\colon \widehat E \to \widehat E \otimes \widehat \Omega_{\widehat X/W}^1(\log \widehat D)$ is a formal integrable logarithmic connection relative to $W$.

\begin{proposition}
The Frobenius-pullback functor $F^*$ on $\Conn(\widehat{X}/W)$ can be extended to an exact functor 
$F^*\colon \Conn(\widehat{X}/W,\widehat{D}/W) \to \Conn(\widehat{X}/W,\widehat{D}/W).$
\end{proposition}
\begin{proof}
Let $(\widehat E,\widehat \nabla)$ be an object in $\Conn(\widehat{X}/W,\widehat{D}/W)$. The Frobenius pullback is already defined in the complement of $D$. For a sufficiently small open subset $\widehat{U} \subset \widehat{X}$ intersecting the normal crossings divisor $\widehat{D}$, we choose a formal Frobenius lift $\widehat {F}$, which preserves the divisor, that is, with the property
\ga{}{ \widehat{ F}^* \mathcal{O}_{\widehat{X}}(-\widehat{D})
\subset \mathcal{O}_{\widehat{X}}(-\widehat{D}).\notag}

It suffices to show that the formula \eqref{eqn:psi} defining the glueing data $\psi_{\widehat{G}\widehat{F}}$ can still be applied in the logarithmic setting. By passing to an appropriate \'etale neighbourhood endowed with a system of local coordinates $(x_i)$ we may assume that $U \cap D$ is defined by the local equation $x_1\cdots x_s = 0$ where $x_i$ defines the irreducible component $U\cap D_i$.

The term $\partial^{\underline{i}}(s)$ is a section of $E(\sum_{j=1}^s i_j D_j)$. However, the right-hand term
$$\frac{\prod_{j=1}^d(\widehat{F}_{j}-\widehat{G}_{j})^{i_j}}{\underline{i}!},$$
is a section of $\Oc(-\sum_{j=1}^s i_j D_j)$, since $\widehat{F}$ and $\widehat{G}$ respect the divisor $D$, and thus the tensor product
 $$\partial^{\underline{i}}(s) \otimes \frac{\prod_{j=1}^d(\widehat{F}_{j}-\widehat{G}_{j})^{i_j}}{\underline{i}!}$$
belongs to $E$. The convergence condition holds as in the case $D=\emptyset$.  Therefore  $\psi_{\widehat{G}\widehat{F}}$ is still well-defined.

Lemma \ref{lemma:ab} still holds, since it suffices to verify these identities in the open dense subset given by the complement of $D$. Indeed  on a vector bundle a connection with logarithmic poles, is uniquely determined by its value on the complement of the pole divisor.
\end{proof}

\begin{lemma}\label{lemma:etale-basechange}
The functor $F^*$ is compatible with base change along \'etale morphisms. That is, for every \'etale morphism $h\colon \widehat X_1 \to \widehat X_2$ there is the following  diagram of functors 
\[
\xymatrix{
\Conn(\widehat X_1,\widehat D) \ar[d]_{h^*} \ar[r]^{F^*} & \Conn(\widehat X_1,\widehat D) \ar[d]^{h^*} \\
\Conn(\widehat X_2,h^*\widehat D) \ar[r]^{F^*} & \Conn(\widehat X_2,h^*\widehat D)
}
\]
commuting up to an invertible natural transformation.
\end{lemma}

\begin{proof}
This follows from the observation that the absolute Frobenius commutes with \'etale pullbacks. Therefore, the pullback of a local Frobenius lift is a local Frobenius lift.   
\end{proof}

In order to define strict normal crossing divisors on  a Deligne-Mumford stack, we follow the convention of \cite[D\'efinition 1.3.17(2)]{Laaroussi}, i.e., each irreducible component is required to be regular. We denote by $\widehat{\Xc}, \widehat{\Dc}$ the formal completions.

\begin{corollary}\label{cor:Cartier-DM}
 Let $\Xc/W$ be a smooth Deligne-Mumford stack over $W$ with a strict normal crossings divisor $\Dc \hookrightarrow \Xc$. Then, there is an exact functor (unique up to a unique natural transformation) 
 $$F^*\colon \Conn( \widehat \Xc, \widehat \Dc) \to \Conn( \widehat \Xc, \widehat \Dc),$$
 such that for every formal scheme $\widehat S$ and every \'etale morphism $h\colon \widehat S \to \Xc$ we have a natural transformation $h^*F^*_{\Xc} 
 \simeq F^{*}_{\widehat{S}} h^*$. 
\end{corollary}

It will also be important for us to know that the Frobenius pullback preserves nilpotent residues.

\begin{lemma}\label{lemma:F_residues}
The Frobenius pullback $F^*$ sends a flat connection $(E,\nabla)$ with nilpotent residues $\res_i \nabla$ to a flat connection with nilpotent residues.
\end{lemma}
\begin{proof}
This follows from the local description, as pulling back along a local Frobenius lift preserves nilpotency of the residues.   
\end{proof}

\subsection{Frobenius pullback and the inverse Cartier transform}

The functor $F^*$ described in the previous section is not an equivalence of categories. However, it is not far off from being one as we  explain below. First we must recall the notion of a $p$-connection.
\begin{definition} \label{def:pconn}
A (flat) $p$-connection (with logarithmic poles along $\Dc$)  $\nabla'$ on a  formal  vector bundle $\widehat V$ on $\widehat{\Xc}$ is a  formal  $W$-linear map  
$\widehat V\to \widehat V\otimes_{\mathcal O_{\widehat \Xc}} \Omega^1_{\mathcal O_{\widehat \Xc}}$ 
($\widehat V\to \widehat V\otimes_{\mathcal O_{\widehat \Xc}} \Omega^1_{\mathcal O_{\widehat \Xc}} (\log \widehat \Dc)$) 
satisfying the $p$-Leibniz rule $\nabla'(fs) = p s \otimes df + f\nabla'(s)$, where $f$ is a regular  formal function and $s$ a local section of $\widehat V$, and satisfying the integrality condition  $(\nabla')^2(s)=0$.  We denote by  $ \pConn(\widehat \Xc,\widehat \Dc)$ the category of $p$-connections with logarithmic poles along $\widehat \Dc$. 
\end{definition}

 We shall show the existence of a factorisation $$F^* = C^{-1}\circ [p].$$ 
 Here  $C^{-1}$ is the so-called \emph{inverse Cartier operator} which is an equivalence of categories and 
$$[p]\colon \Conn( \widehat \Xc,\widehat \Dc) \to \pConn(\widehat \Xc,\widehat \Dc)$$
sends a flat log-connection $(\widehat E,\ \widehat \nabla)$ to the $p$-connection $(\widehat E,p \widehat \nabla)$.

Without log-poles this is a known result, and follows essentially from Xu's lift of the inverse Cartier operator to $W,$ see \cite{Xu}.

Since it would take us too far afield, we will not provide a complete generalisation of Xu's work to the log-case. The construction presented in this article is tailor-made for the proof presented in Section~\ref{sec:topproof}.

\begin{definition}\label{defi:Cartier}
\begin{enumerate}
    \item[(a)] We denote by $\Conn(\widehat{\Xc},\widehat{\Dc})_0$ the full subcategory of $\Conn(\widehat{\Xc},\widehat{\Dc})$ consisting of the 
    $(\widehat E, \widehat \nabla)\  $ for which 
     the special fibre   $(\widehat E, \widehat \nabla)\otimes_W k  $   has zero $p$-curvature.  
    \item[(b)] We denote by $\pConn(\widehat{\Xc},\widehat{\Dc})_0$ the full subcategory of $(\widehat V,\nabla')$ where the special fibre $(\widehat V, \nabla')\otimes_W k $ is the zero Higgs field.
    \item[(c)] There is a functor $[\frac{1}{p}]\colon \pConn(\widehat{\Xc},\widehat{\Dc})_0 \to \Conn(\widehat{\Xc},\widehat{\Dc})$ sending $(\widehat V,\nabla')$ to $(\widehat V,\frac{\nabla'}{p})$.
    \item[(d)]\label{cartier} The functor $C^{-1}\colon \pConn(\widehat{\Xc},\widehat{\Dc})_0 \to \Conn(\widehat{\Xc},\widehat{\Dc})_0$
 is defined to be the composition $$C^{-1}=F^* \circ [\frac{1}{p}].$$\end{enumerate}
\end{definition}

In the absence of log-poles, the following proposition is (a special case of) the main result of \cite{Xu}. We highlight that this result should be expected to hold without imposing the restriction of vanishing $p$-curvature.  It is  however  the only case required for our purposes and  leads to a drastic simplification of the argument.

\begin{proposition}\label{prop:log-Xu}
The functor $C^{-1}=F^*\circ [\frac{1}{p}]\colon \pConn(\widehat{\Xc},\widehat{\Dc})_{0} \to \Conn(\widehat{\Xc},\widehat{\Dc})_{0}$ is fully faithful. Its essential image is given by the full subcategory of flat log-connections $(E,\nabla)$ with $p$-divisible residues.
\end{proposition}

The proof will be given at the end of Subsection \ref{sub:log}.

\subsection{Relative Frobenius morphisms and twists}

{As preparation for the proof of Proposition \ref{prop:log-Xu}, we study the absolute Frobenius morphism
 $F_{\Theta_k}$  of the stack $\Theta_k=[\Ab^1_k/\Gb_{m,k}]$. We begin by recalling the definition of the absolute Frobenius morphism of a stack.}

\begin{definition}
Let $\Yc$ be an $\Fb_p$-stack, or more generally a presheaf on the category $\Sch_{\Fb_p}$. For every $S \in \Sch_{\Fb_p}$ we define 
$$F_{\Yc}\colon \Yc(S) \to \Yc(S)$$
to be the map sending $S \to \Yc$ to the composition $S \xrightarrow{F_S} S \to \Yc$. 
Note that the construction above defines a natural transformation since for every morphism $g\colon S \to T$ in $\Sch_{\Fb_p}$ there is a commutative diagram
\[
\xymatrix{
S \ar[r]^{F_S} \ar[d]_g & S \ar[d]^g \\
T \ar[r]^{F_T} & T.
}
\]
Therefore, it gives rise to an endomorphism of the stack $\Yc$ denoted by $F_{\Yc}$.
 \end{definition}

\begin{definition}
Given a morphism of $\Fb_p$-stacks $\Yc \to \Ss$ we denote by $\Yc^{(\Ss)} \to \Ss$ the base change $\Yc \times_{\Ss,F_{\Ss}} \Ss \to \Ss$, and refer to it as the relative Frobenius twist. The relative Frobenius morphism $Fr_{\Yc/\Ss}\colon \Yc \to \Yc^{(\Ss)}$ is a morphism of $\Ss$-stacks obtained from the universal property of the fibre product: 
\[
\xymatrix{
\ar@/^2.0pc/[rr]^{
F_{\Yc}}\Yc \ar[r]_{Fr_{\Yc/\Ss}} \ar[rd] & \Yc^{(\Ss)} \ar[r] \ar[d] & \Yc \ar[d] \\
& \Ss \ar[r]^{F_{\Ss}} & \Ss.
}
\] 
\end{definition}

\begin{notation}\label{notation:prime}
The Frobenius twist of $\Y\to \Spec k$ will be denoted by $\Yc'$.  
\end{notation} 
In the following, we denote by $\Theta$ the stack $[\Ab^1/\Gb_m]$ over $W$  and by $\Theta_k=[\Ab^1_k/\Gb_{m,k}]$ its mod $p$ reduction. 

{
The following definition of a root stack stems from Abramovich--Graber--Vistoli \cite[Appendix B.2]{AGV} and Cadman \cite{cadman}. 

\begin{definition}\label{defi:root-stack}
    Let $s\colon \bar{X} \to \Theta^m$ be the morphism induced by the irreducible components of the effective divisor $D \subset \bar{X}$. For $n \in \mathbb{N}$ we define $\bar{\mathcal{X}}_{\underline{D},n}= \bar{X} \times_{\Theta^m,[n]} \Theta^m$, where $[n]\colon \Theta \to \Theta$ denotes the map induced by the $n$-th power morphism.
\end{definition}

\begin{remark}
The stack $\Theta=[\Ab^1/\G_m]$ classifies {effective generalised Cartier divisors (we thank M. D'Addezio for sharing this terminology with us), i.e.,} pairs consisting of a line bundle and a section. That is, a morphism $S \to \Theta$ corresponds to a pair $(L,s)$, where $L$ is a line bundle on $S$ and $s \in \Gamma(S,L)$ is a section, see \cite[Example 5.13]{Ol}. Every effective Cartier divisor $D \subset S$ gives rise to the pair $(\Oc_S(D),1)$, where we denote by $1$ the constant regular function in $\Oc_S \subset \Oc_S(D)$.

As in \cite[Appendix B.2]{AGV}, we may therefore describe the root stack as adjoining an $n$-th root to the pair $(L,s)$. That is, an $S$-point of $\bar{\Xc}_{D,n}$ corresponds to a triple $(M,t,f,\phi)$, where $M$ is a line bundle on $S$, $t \in \Gamma(S,M)$, $f\colon S \to \bar{X}$ is a morphism of schemes and 
$\phi\colon M^n \xrightarrow{\simeq} f^*L$ is an isomorphism such that $\phi(t^n)=s$.
\end{remark}

\begin{lemma}\label{lemma:smooth}
The root stack $X_{\underline{D},n}$ is a tame stack in the sense of \cite[Definition 3.1]{AOV}.
Furthermore, if $D=\bigcup_{i=1}^m D_i\subset X$ is a strict normal crossings divisor, with irreducible components, which we assume to be smooth over $W$, then the root stack $\Xc_{\underline{D},n}$ is a smooth algebraic $W$-stack.    
\end{lemma}
\begin{proof}
Tameness follows from the fact that $\mu_n$ is a finite flat linearly reductive group scheme and the presentation of $X_{\underline{D},n}$ as a $\mu_n^m$-quotient (see \cite[Theorem 3.2(c)]{AOV}).

Algebraicity of root stacks holds by virtue of definition as a fibre product of algebraic stacks. We will thus focus on smoothness. Using a system of \'etale coordinates adapted to the divisor $D$, we may assume without loss of generality that $X=\Ab^d_W$ and $D$ the vanishing locus of $z_1\cdots z_m = 0$. By definition, there is an equivalence of stacks
$$\Xc_{D,n} \cong [\Ab_W^1/\mu_n]^m \times \Ab^{d-m}_W.$$
It therefore suffices to establish smoothness of $[\Ab_W^1/\mu_n]$. If $n$ is invertible in $W$, then the group scheme $\mu_n$ is \'etale. Smoothness of the affine line concludes the proof in this case. 

To treat the general case, we rewrite the quotient as follows:
$$[\Ab^1_W/\mu_n]=[(\Ab_W^1 \times \Gb_{W,m})/\Gb_{W,m}],$$
where the $\Gb_m$-action on the right-hand side is given by
$$\lambda\cdot{} (z,w)=(\lambda z,\lambda^{-p}w).$$
Smoothness of the multiplicative group and of the atlas $(\Ab_W^1 \times \Gb_{W,m})$ yields smoothness of the quotient stack.
\end{proof}
}

\begin{lemma}
{The diagram 
\[
\xymatrix{
(\Xc_k')_{\underline{D}',p} \ar[r] \ar[d]_{c_{\Dc'}} & \Xc_k \ar[d]^{c_{\Dc}} \\
\Theta_k^m \ar[r]^{F_{\Theta_k^m}} & \Theta_k^m
}
\]
is cartesian. Here, we employ Notation \ref{notation:prime} and Definition \ref{defi:root-stack}.}
\end{lemma}

\begin{proof}
{Let us denote by $\varphi\colon k \to k$ the Frobenius automorphism $k$.

We may then express $F_{\Theta_k}$ as the composition
$$\Theta_k \xrightarrow{[p]_{\Theta}}  \Theta_k = \Theta_{\Zb} \times \Spec k \xrightarrow{(\id_{\Theta_{\Zb}} \times \varphi)} \Theta_{\Zb} \times \Spec k = \Theta_k,$$ 
where $[p]_{\Theta}\colon \Theta_k \to \Theta_k$ denotes the $p$-th power map of $\Theta_k$ sending a pair $(L,s)$ to $(L^p,s^s)$, and $\Theta_{\Zb}$ refers to the stack $[\Ab^1_{\Zb} / \Gb_{m,\Zb}]$.

This shows that base change along $F_{\Theta^m_k}$ agrees with base change along $\varphi$, corresponding to a Frobenius twist, followed by base change along the map $[p]\colon \Theta^m_k \to \Theta_k^m$, corresponding to the iterated root stack construction.
}
\end{proof}

\begin{lemma}\label{lemma:twist}
A Frobenius lift $\phi\colon W \to W$ gives rise to a lift $F_{\Theta}(\phi)\colon \Theta_W \to \Theta_W $ of $F_{\Theta_k}$. 
\end{lemma}
\begin{proof}
{ We define $F_{\Theta}(\phi)$ to be the composition 
$$\Theta_W \xrightarrow{[p]_{\Theta}}  \Theta_W = \Theta_{\Zb} \times \Spec W \xrightarrow{(\id_{\Theta_{\Zb}} \times \phi)} \Theta_{\Zb} \times \Spec W = \Theta_W,$$ 
where $[p]_{\Theta}\colon \Theta_W \to \Theta_W$ denotes the $p$-th power map of $\Theta_k$.}
\end{proof}

\subsection{Logarithmic connections via stacks}\label{sub:log}
Our approach to Proposition~\ref{prop:log-Xu} relies on an observation due to Martin Olsson (see \cite[Section 9]{Olsson}), according to which the sheaf $\Omega_{\Xc} ^1(\log \Dc)$ is isomorphic to the sheaf of relative K\"ahler forms $\Omega_{\Xc /\Theta^m}$, where $\Theta^m = [\Ab^1/\G_m]^m$ is the stack classifying effective generalised Cartier divisors, and $c_D\colon \Xc\to \Theta^m$ is the classifying map corresponding to the smooth irreducible components $\Dc=\bigcup_{i=1}^m \Dc_i$.  

We would like to emphasise that we learnt about this point of view from \cite[Chapitre 2]{Laaroussi}, which contains an in-depth treatment of Olsson's point of view on logarithmic differential forms.

For the benefit of the reader, we will unravel the definition of $\Omega^1_{\Xc/\Theta}$ until we arrive at a more elementary version of Olsson's isomorphism which does not refer to stacks. In the following example we will omit the subscript $W$ referring to the base, since the choice of base scheme does not play any role in the considerations to follow.  

\begin{example}[Olsson's isomorphism for $m=1$]
Let $c_D \colon \Xc \to \Theta = [\Ab^1/\Gb_m]$ be a morphism classifying an effective Cartier divisor $\Dc$. Let us assume that $\Dc$ is smooth. The base change 
$$\Xc \times_{\Theta} \Ab^1 = \mathcal T$$
agrees with the $\Gb_m$-torsor $\mathcal T\xrightarrow{\tau} \Xc$ given by the total space of the line bundle $\Oc_{\Xc}(\Dc)$ minus the zero section. Note that there is a natural function $f$  given by the projection of 
$\mathcal T$ to  $\mathbb A^1$
which vanishes along
 $\mathcal T \times_{\Xc} \Dc$ of first order.

Olsson's isomorphism amounts to the assertion that there is a $\Gb_m$-equivariant isomorphism of sheaves
$\tau^*\Omega_{\Xc}^1(\log \Dc) \simeq \Omega_{\mathcal T/\Ab^1}^1=:\Omega^1_f,$ 
which implies  $\Omega^1_{\Xc}(\log \Dc) \simeq (\tau_*\Omega_f^1)^{\Gb_m}$. A detailed computation for   $X=\Ab^1$  is given in \cite[Exemple~2.1.6]{Laaroussi}. The general case can be reduced to this one via \'etale descent.
\end{example}

Henceforth, we identify a formal connection $(\widehat{E},\widehat{\nabla})$ with logarithmic poles on $(\widehat{\Xc},\widehat{\Dc})/W$ with a relative connection on the pair $\widehat{\Xc}/\Theta^m$, using Olsson's isomorphism $\Omega^1_{\widehat{\Xc}/\Theta_W^m}= \Omega^1_{\widehat{\Xc}}(\log \Dc)$.
{
As shown in Lemma \ref{lemma:smooth}, the generalised root stack $X^{(\Theta^m)}$ is a smooth algebraic stack. Thus, there is a well-defined theory of flat {($p$-)}connections on this stack, which can be defined in terms of flat ($p$-)connections on a smooth atlas satisfying a descent condition. This notion will play a key role in the remainder of this section. We emphasise that we will not need to extend the Frobenius pullback functor $F^*$ from Corollary \ref{cor:Cartier-DM} to arbitrary smooth algebraic stacks.
In the following we denote by $X^{\phi}$ the twist of $X$ with respect to a fixed Frobenius lift $\phi \colon W \to W.$
\begin{lemma}\label{lemma:mod-p}
Pullback along the map $\rho\colon\widehat{X}^{(\Theta^m)} \to \widehat{X}^\phi$ yields a fully faithful functor
$$\pConn(\widehat{X}^{\phi},D^{\phi})=\pConn(\widehat{X}^{\phi}/\Theta^m)\hookrightarrow \pConn(\widehat{X}^{(\Theta^m)}/\Theta^m).$$
Furthermore, an object $(M,\nabla') \in \pConn(\widehat{X}^{(\Theta^m)}/\Theta^m)$ belongs to $\pConn(\widehat{X}^{\phi}/\Theta^m)$ if and only if its special fibre belongs to $\pConn(X'/\Theta_k^m)$.
\end{lemma}
\begin{proof}
Full faithfullness follows from the fact that $\rho_*\rho^*$ is isomorphic to the identity functor via the unit of the adjunction. To see this one uses tameness of the root stack (see Lemma \ref{lemma:smooth}), which yields exactness of the functor $\rho_*$ (\cite[Definition 3.1 \& Theorem 3.2]{AOV}). The projection formula and the formula $\rho_*\rho^*(\Oc,d) \simeq (\Oc,d)$ now shows that the unit map of the adjunction is an isomorphism.

The second claim follows from the fact that a $\mu_p$-module over $W$ is endowed with a trivial action if and only if the action on the special fibre is trivial.
\end{proof}
}
{
As a next step we extend a result of Shiho to the log-setting \cite{Shiho}. We denote by the subscript $\qnilp$ the full sub-category consisting of flat {($p$-)}connections, for which we assume quasi-nilpotence, as defined in \cite[Definition 1.5]{Shiho}.

\begin{proposition}\label{prop:Shiho-log}
We fix a Frobenius lift $\phi$ for $W$. Assume that $(\widehat{X},\widehat{D})$ is a log-pair consisting of scheme, which is endowed with a relative log-Frobenius lift $\widehat{F}$ compatible with $\phi$, i.e., $\widehat{F}^*|_{W\subset \Oc} = \phi$. Then, there is an equivalence of categories $$\pConn(\widehat{X}^{({\Theta^m})}/ {\Theta^m)}_{\qnilp} \to  \Conn(\widehat{X}/{\Theta^m})_{\qnilp}.$$
\end{proposition}
\begin{proof}
By assumption, the commutative diagram defined over $k$
\[
\xymatrix{
X_k \ar[r]^{F_{X_k}} \ar[d] & X_k \ar[d] \\
\Theta_k^m \ar[r]^{F_{\Theta_k^m}} & \Theta_k^m
}
\]
lifts to the ring $W$ of mixed characteristic. Here, we denote by $F_{Y_k}$ the absolute Frobenius morphism of a $k$-scheme $Y_k$. The lift of this diagram is a diagram of formal $W$-schemes
\[
\xymatrix{
\widehat{X} \ar[r]^{\widehat{F}_X} \ar[d] & \widehat{X} \ar[d] \\
{\Theta}^m \ar[r]^{\widehat{F}_{\Theta^m}} & {\Theta}^m.
}
\]
We recall that the lift of the bottom row is given by a Frobenius lift $F_{\Theta}(\theta)$ of Lemma \ref{lemma:twist}.

Recall that $\widehat{X}^{(\Theta^m)}$ is the fibre product $\widehat{X} \times_{(\Theta^m),\widehat{F}_{\Theta^m}} {\Theta^m}.$ By definition, it is a formal $W$-lift of the relative Frobenius twist of $X_k/\Theta^m_k$. The universal property of the fibre product yields a $W$-morphism 
\begin{equation}\label{eqn:lift}
\widehat{Fr}_{X/\Theta^m}\colon \widehat{X} \to  \widehat{X}^{( {\Theta^m})}
\end{equation}
lifting the relative Frobenius. We use that $\Ab^m \to \Theta^m = [\Ab^1/\Gb_{m}]^m$ is an atlas for the stack $\Theta^m$. Therefore, there is a square of categories \[
\xymatrix{
\pConn(\widehat{X}^{( \widehat{\Ab^m})}/ {\widehat{\Ab^m}})_{\qnilp} \ar[r]^-{\cong}  \ar[d] & \Conn(\widehat{X}/\widehat{\Ab^m})_{\qnilp} \ar[d] \\
\pConn(\widehat{X}^{({\Theta^m})}/ {\Theta^m)}_{\qnilp} \ar[r]^-{C^{-1}}& \Conn(\widehat{X}/{\Theta^m})_{\qnilp}
}
\]
which commutes up to an invertible natural transformation. The functor $C^{-1}$ is defined as in Shiho's article \cite{Shiho} and requires the choice of a formal lift of the relative Frobenius morphism. We use the map provided by \eqref{eqn:lift}. An object $(V,\tilde{\nabla}) \in \pConn(\widehat{X}^{(\Theta^m)}/\Theta^m)_{\qnilp}$ is sent to 
$$(E,\nabla) = \widehat{Fr}_{X/\Theta^m}^*\big(V,\frac{\widehat{\nabla}}{p}\big):=\big(\widehat{Fr}_{X/\Theta^m}^*V,
\frac{\widehat{Fr}_{X/\Theta^m}^*{\widehat{\nabla}}}{p}\big).$$

The top row of this diagram is an equivalence of categories by  the main result of \cite{Shiho}. As an application of faithfully flat descent theory we obtain that the bottom row is an equivalence of categories too.
\end{proof}
\begin{remark}
The formulation of Proposition \ref{prop:Shiho-log} involves a generalised root stack, which possibly renders the statement somewhat obscure. We emphasise that the category of vector bundles with $p$-connections on a root stack admits an explicit description in terms of parabolic structures (see for example \cite{TV}).
\end{remark}
 Below, we give an example to explain why one should expect either root stacks or parabolic structures to arise. 
\begin{example}
Let $(X_0,D_0)/k$ be an arbitrary log-pair, with $D=\bigcup D_i$ being a strict normal crossings divisor with smooth irreducible components $D_i$. We fix an index $j$. The trivial flat connection $(\Oc_X,d)$ can be restricted to the subsheaf $\Oc(-D_j)$, and produces a flat log-connection $\nabla_j$ on $\Oc(-D_j)$ with zero $p$-curvature. Due to vanishing of the $p$-curvature we expect $(\Oc(-D_j),\nabla_j)$ to be a Frobenius pullback. However, this would produce a $p$-th root of the line bundle $\Oc(-D_j)$.  By construction, such a root exists on the generalised root stack $X^{(\Theta^m)}$. In the theory of parabolic flat connections, a $p$-th root can also be encoded by means of adding a fractional parabolic weight along $D_j$.
\end{example}

\begin{lemma}\label{lemma:Shiho}
Proposition \ref{prop:log-Xu} holds for an affine formal scheme $\widehat X$.
\end{lemma}
\begin{proof}
It follows from Lemma \ref{lemma:twist} that $\widehat{X}^{(\Theta^m)}$ corresponds to { a root stack over the log-pair $(\widehat{X}^{\phi},\widehat{D}^{\phi})$, given by the same underlying log-pair $(X,D)$ with $W$-structure twisted by the chosen Frobenius lift $\phi$. In particular, there is an embedding (see Lemma \ref{lemma:mod-p})

$$\pConn(X^{\phi},D^{\phi}) \hookrightarrow \pConn(X^{(\Theta^m)}/\Theta^m).$$}

Using the identification  $(\widehat X,\widehat D) \simeq (\widehat X',\widehat D')$, provided by this isomorphism (here, perfectness of $k$ is   used), henceforth called $\sigma$, we obtain an equivalence 
$$C^{-1}\circ \sigma^{-1}\colon \pConn(\widehat{X}^{( {(\Theta^m})}/ {\Theta^m})_0 \cong \Conn(\widehat{X}/ {\Theta^m})_0.$$
Since the composition $\sigma^{-1}\circ \widehat{Fr}_{X/\Theta^m}$ is a lift of the absolute Frobenius morphism of $X_k$ (compatible with $D_k$), we see that $C^{-1}\circ \sigma^{-1}$  agrees with the functor $F^*\circ[\frac{1}{p}]$ of Proposition~\ref{prop:log-Xu}. 

{We therefore see that $F^* \circ [\frac{1}{p}]$ is given by the composition of functors
$$\pConn(X^{\phi},D^{\phi}) \hookrightarrow \pConn(\widehat{X}^{( {(\Theta^m})}/ {\Theta^m})_0  \cong \Conn(\widehat{X}/ {\Theta^m})_0.$$
To conclude the proof of the lemma it suffices to identify the essential image with the sub-category of flat log-connections with $p$-divisible residue.

An object in the image of this functor is a Frobenius pullback $F^*(E,\nabla)$, and therefore the residues vanish modulo $p$. Vice versa, if $(E,\nabla)$ has $p$-divisible residues and zero $p$-curvature, there exists $(M,\nabla') \in \pConn(\widehat{X}^{(\Theta^m})/\Theta^m)$, such that $$C^{-1}(M,\nabla') \simeq (E,\nabla).$$ 

The special fibre $(M_0,\nabla_0') \in \pConn(\widehat{X}_k^{(\Theta_k^m}))/\Theta_k^m)$ still satisfies the property $$C^{-1}(M_0,\nabla'_0) \simeq (E_0,\nabla_0),$$
and furthermore, the $p$-connection $(M_0,\nabla'_0)$ with this property is unique up to a unique isomorphism since $C^{-1}$ is an equivalence of categories. On the other hand, recall that by assumption in Proposition~\ref{prop:log-Xu}, 
$(E_0,\nabla_0)$ has  vanishing $p$-curvature. 
 
Applying Cartier descent to the connection (without log-poles) $(E_0,\nabla_0)$, we obtain
$$F^*(E_0^{\nabla_0}) \cong (E_0,\nabla_0).$$
Thus, there is an isomorphism $(M_0,\nabla'_0) \simeq C^{-1}(E_0^{\nabla_0},0).$
The right-hand side belongs to the sub-category $\pConn(X'_0/k) \subset \pConn(X_0^{(\Theta^m_k)}/\Theta^m_k)$. By  Lemma \ref{lemma:mod-p} this allows one to conclude that $(M,\nabla') \in \pConn(\widehat{X}^{(\Theta^m)}/\Theta^m)$.
}
\end{proof}
}

\begin{proof}[Proof of Proposition \ref{prop:log-Xu}]
The assertion that a given functor between sheaves of categories is an equivalence can be verified \'etale locally. We may therefore reduce everything to the situation where a logarithmic Frobenius lift exists and apply Lemma \ref{lemma:Shiho}.
\end{proof}

\begin{remark}
Xu's lift of the inverse Cartier operator is currently only available in an absolute setting. Indeed, the main result of \cite{Xu} is an equivalence of categories $\Conn(\widehat{X}/W)^{\qnilp} \cong \pConn(\widehat{X}/W)^{\qnilp}$. 
The superscript  $\qnilp$  stands for ``quasi-nilpotent'' and  indicates that the $p$-curvature of $\nabla_k$ is nilpotent of an appropriate level, see \cite[2.6]{EG20} for further details.
 In order to apply the argument of the proof of Lemma~\ref{lemma:Shiho} to prove Proposition~\ref{prop:log-Xu}, a relative version of Xu's result would be required. This would then immediately yield a logarithmic version of the same result, without the restrictive assumption of vanishing $p$-curvatures.
\end{remark}

{ We conclude this section by remarking that over $k$, a logarithmic analogue of the Ogus-Vologodsky correspondence was developed by Schepler \cite{Schepler}. The treatment given in \emph{loc. cit.} avoids the theory of generalised root stacks and is instead based on the formalism of indexed algebras.}

\subsection{The flow functor $\Phi$} \label{sec:Phi}

The purpose of this section is to present a modified version of $F^*$ for filtered flat connections. The resulting functor is merely a convenient tool to recast the Higgs-de Rham flow of \cite{LSZ} and is closely related to the theory of Fontaine--Lafaille modules, see \cite{Faltings}.
We remind the reader of our standing assumption $p \geq 3$.

\begin{definition}
We denote by $\FConn(\widehat{X},\widehat{D})$ the category of triples $(\widehat E, \widehat \nabla,\Fil)$, where $(\widehat E,\widehat \nabla)$ is a formal  vector bundle endowed with a formal flat log-connection, and $\Fil$ is a \emph{Griffiths-transverse} descending filtration by locally free locally split subsheaves, i.e., $$\nabla(\Fil^i) \subset \Fil^{i-1} \otimes \Omega_{\widehat{X}/W}^1(\log \widehat D)$$  and $gr^{\Fil}$ is locally free.
\end{definition}

There is a well-known construction allowing us to pass from a filtered flat connection to a $p$-connection (see Definition~\ref{def:pconn}), which is closely related to the Artin-Rees construction.

\begin{definition}[Artin-Rees construction] \label{dfn:AR}
\begin{enumerate}
    \item[(a)] The Artin-Rees construction $\AR_{p}^{\Fil^{\bullet}}( \widehat E) = \widetilde{E}$ is the sheaf of $\Oc_X$-sub-modules 
    $$\sum_{i\in \Nb} \Fil^i \otimes p^{-i} \subset \widehat E \otimes_W K$$ where multiplication by $p$ maps the term $\Fil^i \otimes p^{-i}$ to $\Fil^{i-1} \otimes p^{-i+1}$ via the  inclusion $\Fil^i 
   \xrightarrow{ 
    \subset}  \Fil^{i-1}$  on the left and $p^{-i} \xrightarrow{\cdot p} p^{-i+1}$ on the right. The $p$-connection $p \widehat \nabla$ on $\widehat E \otimes_W K$ stabilises $\AR^{\Fil^{\bullet}}_{p}(\widehat E) \subset \widehat E\otimes_WK $ and defines a $p$-connection on it denoted by $\widetilde{\nabla}$.
     The resulting functor is denoted by
    $$\AR_{p}\colon \FConn \to \pConn.$$
    \item[(b)] There is a canonical map $\bigoplus_i \Fil^i \to \AR^{\Fil^{\bullet}}_{p}(\widehat E)$ sending $\Fil^i$ to $\Fil^i \otimes p^{-i}$ via the map $s \mapsto s \otimes p^{-i}$. For $s \in \Fil^i$ we follow the standard convention to denote its image in $\AR_{p}( \widehat E)$ by $(s)_i$.
\end{enumerate}
\end{definition}

The following definition of the flow functor $\Phi$ builds up on the $p$-connection defined above.
\begin{definition} \label{dfn:Phi}
\begin{enumerate}
\item[(a)] For each affine subscheme endowed with a Frobenius lift $\widehat{F}$ we define 
$$\Phi(\widehat E, \widehat \nabla,\Fil):= \widehat V=\widehat{F}^*\AR^{\Fil^{\bullet}}_{p}(\widehat E)$$ endowed with the flat connection $\Phi( \widehat \nabla)$ which assigns to a local section $s$ of $\tilde{F}^{-1}\widehat V$ the section
$$\frac{\id \otimes \widehat{F}^*}{p}\widetilde{\nabla}(s).$$
 
\item[(b)] These locally defined objects (defined with respect to the choice $\widehat{F}$ and $\widehat{G}$ of two Frobenius lifts) are glued together with respect to the explicit cocycle 
\begin{multline}\label{eqn:epsilon}
\varepsilon_{\tilde{G}\tilde{F}}\big((s)_i \otimes 1\big) =\\
 \sum_{|\underline{i}| \leq i} \big(\partial^{\underline{i}}(s)\big)_{i-|i|} \otimes \frac{\prod_{j=1}^d(\widehat{F}_{j}-\widehat{G}_{j})^{i_j}}{p^{|\underline{i}|}\underline{i}!} + \\  \sum_{|\underline{i}|> i} p^{|\underline{i}|-i}\big(\partial^{\underline{i}}(s)\big)_0 \otimes \frac{\prod_{j=1}^d(\widehat{F}_{j}-\widehat{G}_{j})^{i_j}}{p^{|\underline{i}|}\underline{i}!},
\end{multline}
where $s$ is assumed to be a local section of $\Fil^i$.
\end{enumerate}
\end{definition}

Equation \eqref{eqn:epsilon} originates from Faltings's article \cite[p. 49]{Faltings}, where it is referred to as the \emph{Taylor formula}. As $\widehat E$ is locally free, the cocycle identity  can be checked away of $\widehat D$. 
 It is then a consequence of \cite[Equation~(13.5.6)]{Xu}.

\medskip

 From now on and until the end of the article, we make the following assumption:\medskip

{\bf $\big(\mathbf{\ast}\big)$:} All Frobenius lifts $\widehat{F}$ chosen have the  following extra property: for $X$ affine, let us denote by $F_{W_2}$ the lift to $X_{W_2}= \widehat X_W\otimes_WW_2$. Then  we request
\ga{}{F_{W_2}^*\mathcal{O}_{X_{W_2}}(- D_{W_2})= \mathcal{O}_{X_{W_2}}(- pD_{W_2}).\notag}
See for example \cite[Corollary~9.10]{EV92} for the feasibility of this assumption.  This enables us to apply \cite[Proof~10.7]{EV92}.

\medskip

Given $(\widehat E, \widehat \nabla, \Fil)$ as in Definition~\ref{dfn:AR}, we denote by 
\ga{}{(\mathcal H ,\theta)=\oplus_{i=0}^w ( gr^{\Fil} E_k, gr^{\Fil}\nabla_k) \notag}
the  associated Higgs bundle over $X_k$. Here $0\le w\le r$ where $w$ is the width of $\Fil$, so $\Fil^0 \widehat E= \widehat E, \Fil^{w+1}\widehat  E=\{0\}$ and $r$ is the rank of $\widehat E$. Recall from \cite[Theorem~2.8]{OV07} that if $w<p$ there is a Cartier inverse functor $C^{-1}(\mathcal{H}, \theta)$ which is a flat connection on a vector bundle of rank $r$.  

\begin{proposition} \label{prop:PhiC}
Assuming  $(\ast)$ and $w<p$, then
 $\Phi( \widehat E, \widehat \nabla, {\rm Fil})_k$  is isomorphic to $C^{-1}(\mathcal H, \theta)$ up to sign of $\theta$. In particular, if $(\mathcal H, \theta)$ is (semi-)stable, so is $\Phi( \widehat E,\ \widehat \nabla, {\rm Fil})_k$.

\end{proposition}
\begin{proof}
We refer to \cite[Corollary~5.10]{Lan14} for the second part. 
Strictly speaking, {\it loc. cit.} consider Higgs bundles and connections without poles, but the proof with log poles and nilpotent residues is the same, mutatis mutandis.
We now address the first part. We first prove the following claims. 

\begin{claim} \label{claim:higgsmodp}
 ${\rm AR}_p (\widehat E, \widehat 
 \nabla, \Fil)/p$ is the Higgs bundle $(\mathcal H, \theta)$. 

 \end{claim}
\begin{proof}

It holds
\ga{}{ \tilde E/p \tilde E= \bigoplus_{i\in \N} (\Fil^i /  \Fil^{i+1})_k   \notag  }
where in this sum over $\BN$ the terms $i > w$ die.
So 
\ga{}{  \tilde E/p \tilde E=\oplus_{i=0}^{w}  gr_i^{\Fil} E_k. \notag}
 Griffiths's transversality
$$\widehat \nabla: \Fil^i\otimes p^{-i}\to \Omega^1_{\widehat X} (\log \widehat D) \otimes \Fil^{i-1}\otimes p^{-i}$$
holds, so
\ga{}{p \widehat \nabla: \Fil^i\otimes p^{-i}\to \Omega^1_{\widehat X}(\log \widehat D) \otimes \Fil^{i-1} \otimes p^{-i+1} .\notag}
This yields
\ga{}{p \widehat \nabla|_{{ gr}_i^{\Fil} E_k  }:  { gr}_i^{\Fil} E_k \to \Omega^1_{X_k} (\log D_k) \otimes gr_{i-1}^{\Fil} E_k \notag}
and thus
\ga{}{ p \widehat \nabla \ {\rm on} \ { gr}_i^{\Fil} E_k = \theta .\notag}
\end{proof}

On the other hand, we have the following claim. 
\begin{claim} \label{claim:modp}
The mod $p$ reduction of  formula~\eqref{eqn:epsilon}  reads

\begin{gather} \label{eqn:epsmodp}
 \varepsilon_{\widehat{G}\tilde{F}}\big((s)_i \otimes 1\big)= \sum_{|\underline{i}| \leq i} \big(\partial^{\underline{i}}(s)\big)_{i-|i|} \otimes \frac{\prod_{j=1}^d(\widehat{F}_{j}-\widehat{G}_{j})^{i_j}}{p^{|\underline{i}|}\underline{i}!}.
\end{gather}

\end{claim}
\begin{proof}
This is just because  the last summand in \eqref{eqn:epsilon}  is divisible by $p$ as  $|\underline{i}|> i$.
\end{proof}
We now finish the proof of Proposition~\ref{prop:PhiC}. 
Under $(\ast)$ 
  $(h_{\alpha \beta}, \xi_\alpha)$ from \cite[Lemma~2.1]{LSZ15} 
is explicitly written in 
\cite[p.110]{EV92} and yields 
\ga{}{ h_{\alpha \beta}=\frac{(\widehat{F}-\widehat{ G})}{p} \notag}
viewed as a  linear map from $\Omega^1_{X}\to \mathcal{ O}_{X}$ (it is naturally defined as a derivation from $\mathcal{O}_{X}\to \mathcal{O}_{X}$ )
 on the two by two intersections of opens, and
 \ga{}{   \xi_\alpha=\frac{\widehat{F}}{p},  \ \xi_\beta=\frac{\widehat{G}}{p}. \notag }

\medskip

In the sequel we use the notation $C^{-1}$ for the Cartier inverse used in \cite[Theorem~2.8]{OV07} and $C^{-1}_{\rm exp}$ for the Cartier inverse defined in \cite[2.2]{LSZ15}. 
By \cite[Section~3]{LSZ15}, $C^{-1}=C_{\rm exp}^{-1}$ so we replace $C$ by $C_{\rm exp}$. By \cite[2.2]{LSZ15},
$ C_{\rm exp}^{-1} (\mathcal{H}, \theta)$
 has the same local underlying vector bundle as $\Phi(\widehat E, \widehat \nabla, {\rm Fil})$   
 and the same local connection by condition (a). The only difference between the  glueing condition  \eqref{eqn:epsmodp} with the glueing condition
  \ga{}{ \sum_{i=0}^{p-1} (\frac{\widehat{F}}{p}- \frac{\widehat{G}}{p})^i(F^*\theta)/i! \notag}
    in \cite[2.2]{LSZ15} is that in the former the summation is over $|\underline{i}|\le i$ while in the latter it is over $|i|\le p-1$.  As already mentioned, $w\le r<p$. In particular, as $\theta$ shifts ${gr}^{\Fil}_i  E_k$ to ${gr}^{\Fil}_{i-1} E_k$,  the summation stops at $|i| \le i$ on $(s_i)\otimes 1$.  This finishes the proof.
\end{proof}

\begin{lemma}\label{lemma:Phi_residues}
The flow functor $\Phi$ sends a filtered flat connection  $(\widehat E, \widehat \nabla,\Fil)$ with nilpotent residues $\res_i \nabla$ to a flat connection with nilpotent residues.
\end{lemma}

\begin{proof}
This follows from the local description, the Artin-Rees construction preserves nilpotency of the residues, and likewise for pulling back along a local Frobenius lift.    
\end{proof}

\subsection{The quasi-unipotent setting}

This subsection contains a (mild) generalisation of the logarithmic analogue of Xu's inverse Cartier operator. Previously, we operated in the setting of logarithmic flat connections on a smooth variety $\bar{X}$ with poles along a normal-crossings divisor $D \subset \bar{D}$, such that the residues $\res_i \nabla$ along the irreducible components $D_i \subset D$ are nilpotent.
We will extend this to the case where $\bar{X}$ is replaced by a root stack (see Definition \ref{defi:root-stack}), which over the complex numbers corresponds to working with local systems with quasi-unipotent monodromy at infinity. 

The root stack approach is well-known, and is equivalent to working with parabolic flat connections with rational weights. However, we will not utilise the parabolic perspective in this paper.

{ The following invertibility condition on $n$, as spelled out in the following lemma below, will be imposed frequently.} In practice this amounts to $n$ being coprime to the residual characteristic $p$.  Nonetheless, we emphasise the role played by root stacks violating this assumption in the proof of Lemma~ \ref{lemma:Shiho}. 

\begin{lemma}
    If $n$ is invertible on $\bar{X}$, then $\bar{\mathcal{X}}_{D,n}$ is a smooth Deligne-Mumford stack endowed with a natural morphism to $\bar{X}$ which is an isomorphism away from $D$.
\end{lemma}
\begin{proof}
This is the content of Theorem 2.3.3 in \cite{cadman}. We remark that we verified smoothness without invertibility condition on $n$ in Lemma \ref{lemma:smooth}.
\end{proof}

Recall that we fixed a smooth scheme $X/W$ with a good compactification denoted $\bar{X} / W$ and boundary divisor $D/W$ (with strict relative normal crossings).

\begin{definition} Let us fix a positive integer $n$, which is coprime to $p$.
\begin{enumerate}
    \item[(a)] We denote by $\Dc \subset \bar{\Xc}_{D,n}$ the normal crossings divisor given by the pre-image of $D \subset \bar{X}$.
    \item[(b)] The stack of families of logarithmic flat connections of rank $r$
    on the root stack $\bar{\Xc}_{D,n} / W$ with nilpotent residues along $\mathcal{D}$ will be denoted by $\Mc_{dR} / W$. 
 \end{enumerate}
\end{definition}



\begin{remark}\label{rmk:linear}
    The stack $\Mc_{dR}$ is an ascending union of open substacks which are linear quotients (see Definition \ref{defi:linear-quotient}). To see this, one uses that the forgetful map $\Mc_{dR} \to \Bun_r(\bar{\Xc}_{D,n})$ is representable. It is well-known that $\Bun_r(\bar{\Xc}_{D,n})$ can be exhausted by linear quotient stacks, see for example \cite[Proposition 4.1.3]{Wang}. The assumption of \emph{loc. cit.} that the base scheme is a field is not needed for the proof of Proposition~4.1.3 in {\it loc. cit.}
\end{remark}

\subsection{The quasi-unipotent Higgs-de Rham flow}\label{sub:quasi}

The flow functor $\Phi$ introduced in Section~\ref{sec:Phi}
allows us to state the definition of the quasi-logarithmic Higgs-de Rham flow.

According to Lemma \ref{lemma:Phi_residues} the following correspondence is well-defined.

\begin{definition}[HdR flow]
The Higgs-de Rham flow (abbreviated as HdR flow) is defined to be the correspondence $\Phi$ given by the graph $\Gamma_{\Phi}\subset (\Mc_{dR}(W)^{\iso})^2$
$$\big\{ \big((E_0,\nabla_0),(E_1,\nabla_1)\big) \big| \exists \Fil^{\bullet} \in \Fil^{\nabla}(E_0,\nabla_0)\colon \Phi(E_0,\nabla_0,\Fil) \simeq (E_1,\nabla_1)  \big\}.$$
Here, we denote by $\Fil^\nabla(E_0,\nabla_0)$ the set of Griffiths-transverse filtrations on $E_0$.
\end{definition}
If there is a filtration $\Fil$ such that $ \Phi(E_0,\nabla_0,\Fil) \simeq (E_1,\nabla_1)$ , we simply write 
$(E_1,\nabla_1)\in \Phi((E_0,\nabla_0))$.

We emphasise that our definition is inspired by the work of Lan--Sheng--Zuo \cite{LSZ}, where a slightly different but equivalent approach was taken to define the flow.
\begin{definition}
An element $(E,\nabla)$ of $\Mc_{dR}(W)^{\iso}$ is called \emph{$f$-periodic}, where $f > 0$ is an integer, if there exists a sequence $(E_i,\nabla_i)_{i=0,\dots,f}$ such that $$\big((E_i,\nabla_i),(E_{i+1},\nabla_{i+1})\big)\in \Gamma_{\Phi}\\
\ {\rm and} \  (E_0,\nabla_0) = (E_f,\nabla_f).$$
\end{definition}

\subsection{The Higgs-de Rham flow models and the Frobenius pullback}

The Higgs-de Rham flow provides an integral model of the Frobenius pullback for flat connections, as discussed below.

In the following we denote by $\Mc_{dR}(K)_W$ the image of the map $\Mc_{dR}(W)$ in $ \Mc_{dR}(K)$:
 $$ \Mc_{dR}(W)  \twoheadrightarrow \Mc_{dR}(K)_W \subset \Mc_{dR}(K).$$

\begin{proposition}[HdR-Frobenius comparison]\label{prop:Phi-F-comparison}
There is a commutative diagram
 \[
\xymatrix{
\Mc_{dR}(W) \ar[r]^{\Phi} \ar[d] & \Mc_{dR}(W) \ar[d] \\
\Mc_{dR}(K)_W \ar[r]^{F^*} & \Mc_{dR}(K)_W, 
}
\]   
i.e., for $(E_0,\nabla_0)$ and $(E_1,\nabla_1)$ in $\Mc_{dR}(W)$ such that $(E_1,\nabla_1) \in \Phi(E_0,\nabla_0)$,  it holds
$$F^*(E_0,\nabla_0)_K \simeq (E_1,\nabla_1)_K.$$
\end{proposition}
\begin{proof}
This follows directly by comparing the Taylor formula \eqref{eqn:psi} used to define $F^*$ to Faltings's Taylor formula \eqref{eqn:epsilon} appearing in the definition of $\Phi$. The two formulae are equivalent upon inverting $p$.
\end{proof}



\begin{corollary}\label{cor:W-model}
 Assume that $(E,\nabla)_K \in \Mc_{dR}(K)_W$ is an $F$-isocrystal of period $f > 0$, that is
 $(F^*)^f(E,\nabla)_K \simeq (E,\nabla)_K.$ Furthermore, suppose that for every $0 \leq i \leq f-1$ there exists a stable $W$-model of $(F^*)^i(E,\nabla)_K$ endowed with a Griffiths-transverse filtration (also defined over $W$). Then, $(E,\nabla)_K$ possesses an $f$-periodic $W$-model.
\end{corollary}
\begin{proof}
A stable $W$-model is unique if it exists,  this assertion is  tantamount to separatedness of the moduli space of stable flat log-connections. We can therefore consider the sequence of stable models $(E_i,\nabla_i)_W$ and use that by Proposition \ref{prop:Phi-F-comparison} it must be $f$-periodic. We conclude that $(E_0,\nabla_0)_W$ yields an $f$-periodic HdR-flow.
\end{proof}

\section{A topological proof of the $F$-isocrystal property} \label{sec:topproof}

As in Subsection \ref{sub:quasi} we fix $X/W$ with good compactification $\bar{X} / W$ and boundary divisor $D/W$ (with normal crossings).
We emphasise that we do not impose any stability conditions and thus consider the stack of all families of logarithmic flat connections on vector bundles of rank $r$  and zero rational Chern classes on the root stack $\Xc_{D,n}$.

The $W$-stack $\Mc_{dR}$ has the curious property that its base change to $K$ is of finite type, whereas the special fibre over $k$ is only locally of finite type and not quasi-compact. Indeed, the existence of an integrable connection on a locally free sheaf in characteristic $p>0$ does not imply that its numerical Chern classes vanish.

However, by bounding Harder--Narasimhan types $\eta$ of the underlying vector bundle one can exhaust $\Mc_{dR}$ by open substacks $\Mc_{dR}^{\eta}$, which are of finite type and are also linear quotient stacks (see Remark \ref{rmk:linear}).

\begin{theorem}
The map 
$F^*\colon \Mc_{dR}(W)^{\iso} \to \Mc_{dR}(W)^{\iso}$ is an open embedding.
\end{theorem}

The proof of this theorem will occupy the rest of this section.

\subsection{Continuity with respect to the $p$-adic topology}

First we analyse continuity of the Frobenius pullback operation.

\begin{proposition}\label{prop:continuous} The 
 map of sets
$F^*\colon \Mc_{dR}(W)^{\iso} \to \Mc_{dR}(W)^{\iso}$ is $p$-adically continuous.
\end{proposition}
\begin{proof}
We fix bounds $\eta_0$ and $\eta_1$ for the Harder--Narasimhan types such that 
$$F^*(\Mc_{dR}^{\eta_0}(W)^{\iso}) \subset \Mc_{dR}^{\eta_1}(W)^{\iso}.$$
 This is possible as the  rational Chern classes
  of a vector bundle on $\widehat{X}$   depend only on $ E_k$, and $c_i(F^*E_k)=p^{i}c_i(E_k)$.  

The Frobenius pullback $F^*$ can also be defined for $W_n$-families of flat connections if $p\ge 3$.  The resulting pullback construction will be denoted by $F_n^*$. We therefore see that there is a continuous diagram
\[
\xymatrix{
\Mc^{\eta_0}_{dR}(W)^{\iso} \ar[r]^{F^*} \ar[d] & \Mc^{\eta_1}_{dR}(W)^{\iso} \ar[d] \\
\varprojlim \Mc^{\eta_0}_{dR}(W_n)^{\iso} \ar[r]^-{\varprojlim F_n^*} & \Mc^{\eta_1}_{dR}(W_n)^{\iso}
}
\]
Continuity of $F^*$ with respect to the $p$-adic topology now follows from Corollary \ref{cor:continuous}.    
\end{proof}

\subsection{Injectivity}

The next step is to show that Frobenius pullback gives rise to an embedding.

\begin{lemma}\label{lemma:injective}
The Frobenius pullback $F^*$ is injective on isomorphism classes of $W$-families of flat connections.
\end{lemma}
\begin{proof}
According to Definition \ref{defi:Cartier}(d), there is an isomorphism between $F^*(E,\nabla)$ and $C^{-1}(E,p\nabla)$. That is, $F^* \simeq C^{-1} \circ [p]$, where $[p]$ denotes the functor $(E,\nabla) \mapsto (E,p\nabla)$.

The functor $C^{-1}$ is { fully faithful by}  Proposition \ref{prop:log-Xu} and thus acts injectively on isomorphism classes. Likewise, $[p]$ is faithful, since the connection $\nabla$ can be recovered from $p\nabla$. We therefore infer that the composition $F^*$ acts injectively on isomorphism classes.
\end{proof}

\subsection{Openness of the image}
A crucial property for our argument is established in the following result.

\begin{theorem}\label{thm:open}
The continuous map $F^*\colon \Mc_{dR}(W)^{\iso} \to \Mc_{dR}(W)^{\iso}$ is open with respect to the $p$-adic topology. That is, for every $p$-adically open subset $U \subset \Mc_{dR}(W)^{\iso}$ the image $F^*(U)$ is $p$-adically open.    
\end{theorem}
\begin{proof}
We fix bounds $\eta_0$ and $\eta_1$ for the Harder--Narasimhan types such that 
$$F^*(\Mc_{dR}^{\eta_0}(W)^{\iso}) \subset \Mc_{dR}^{\eta_1}(W)^{\iso}.$$
 
By openness of $\Mc_{dR}^{\eta_i}\subset \Mc_{dR}$ it suffices to prove that the map 
$F^*|_{\Mc_{dR}^{\eta_0}(W)^{\iso}}$
is a homeomorphism onto its image, and that this image is an open subset of $\Mc_{dR}^{\eta_1}(W)^{\iso}$.

The first property follows from the fact that 
$F^*|_{\Mc_{dR}^{\eta_0}(W)^{\iso}}$
is continuous (see Proposition \ref{prop:continuous}) and also an injective map (see \ref{lemma:injective}) between compact Hausdorff spaces.
It therefore remains to prove that 
$$F^*(\Mc_{dR}^{\eta_0}(W)^{\iso})\subset \Mc_{dR}^{\eta_1}(W)^{\iso}$$
is $p$-adically open.
This amounts to verifying the following assertion.

\begin{cl}
For every $(E_0,\nabla_0) \in \Mc_{dR}^{\eta_0}(W)^{\iso}$ there exists an open neighbourhood $\mathcal V$ of $F^*(E_0,\nabla_0)$ such that for every $(E_1,\nabla_1)$ in $\mathcal V$ there is $(E_2,\nabla_2) \in \Mc_{dR}^{\eta_0}(W)^{\iso}$ with 
$(E_1,\nabla_1) \simeq F^*(E_2,\nabla_2)$.
\end{cl}

By construction, $F^*(E_0,\nabla_0)|_k$ has vanishing $p$-curvature { and residues}. There exists an open neighbourhood $\mathcal V$ such that all $(E_1,\nabla_1) \in \mathcal V$ have isomorphic special fibres, and thus have in particular vanishing $p$-curvature { and residues}.

The identity $F^*(E_0,\nabla_0) \simeq C^{-1}(E_0,p\nabla_0)$ leads us to consider the \emph{ansatz} 
$$(E_2,\nabla_2) = \frac{1}{p}\cdot C(E_1,\nabla_1),$$
that is, to divide the $p$-connection $C(E_1,\nabla_1)$ by $p$ and to apply the inverse equivalence $C$ existing by Proposition \ref{prop:log-Xu} { under the assumptions of vanishing $p$-curvature and residues of the special fibre}. The fact that $(E_2,\nabla_2)$ is well-defined provided that $\mathcal V$ is chosen sufficiently small follows from Lemma \ref{lemma:open} below proving that the map $(E,\nabla) \mapsto (E,p\nabla)$ is open.
\end{proof}    

Below we denote by $\Mc_{p{-}dR}$ the algebraic stack of rank $r$ vector bundles on $\Xc=\Xc_{D,n}$  endowed with a flat $p$-connection $\nabla'$ with nilpotent residues. 

\begin{lemma}\label{lemma:open}
The map $[p]\colon \Mc_{dR}(W)^{\iso} \to \Mc_{p{-}dR}(W)^{\iso}$
induced by multiplying a connection $\nabla$ with $p$ is continuous and open.  
\end{lemma}
\begin{proof}
Since the map $[p]$ (sending a connection $\nabla$ to $p\nabla$) is a morphism of algebraic stacks, it induces a $p$-adically continuous map by the  functoriality of the $p$-adic topology.


By definition, a $p$-connection $(E,\nabla')$ lies in the image if and only if its mod $p$ reduction $(E,\nabla')|_k$ is a Higgs bundle with vanishing Higgs field (i.e., $\nabla'|_k=0$). Therefore, the image of $[p]$ agrees with the pre-image in the top left corner of the diagram below:
\[
\xymatrix{
(\pi_n\circ \beta)^{-1}(\mathsf{Bun}(k)^{\iso})\ar[r]\ar[d]& \Bun(k)^{\iso}\ar[d]\\
\Mc_{p-dR}(W)^{\iso}\ar[r]&\Mc_{p-dR}(k)^{\iso}
}
\]
Here, we denote by $\Bun$ the stack of rank $r$ vector bundles on $X$. By virtue of the pro-finite nature of the $p$-adic topology (see for instance Lemma \ref{lemma:basis}), this pre-image is open. On this open subset one has a continuous inverse $[\frac{1}{p}]$, which sends a $p$-divisible $p$-connection $p\nabla$ to $\nabla$.
\end{proof}

\subsection{Conclusion}

According to Theorem \ref{thm:open}, the self-map $F^*$ of $\Mc_{dR}(W)^{\iso}$ is open. Thus, it sends an isolated $W$-point to an isolated point $W$-point.

\begin{corollary}\label{cor:F-isolated}
Let $[(E,\nabla)] \in \Mc_{dR}(W)^{\iso}$ be an isolated  
point. Then, $F^*[(E,\nabla)]\in\Mc_{dR}(W)^{\iso}$ is also an isolated $W$-point.    
\end{corollary}

To conclude the proof of the $F$-isocrystal property it is necessary to investigate whether a $W$-point which gives rise to an isolated $K$-point is also isolated as a $W$-point.
This is the content of the following two lemmas, the proof of which originates from joint work of the second author with Wyss and Ziegler in the context of $p$-adic integration.

{ In order to apply this lemma, we must rigidify the stack $\Mc_{dR}$ in the sense of \cite[Appendix A]{AOV}, that is, remove the scalar automorphisms $\Gb_m \subset \mathsf{Aut}(E,\nabla)$. 

Note that the topological space $\Mc_{dR}(W)^{\iso}$ is the same whether or not $\Mc_{dR}$ is rigidified, since $W$ has trivial Picard group. However, rigidification affects the representable locus mentioned in the following lemma.}

\begin{lemma}\label{lemma:discrete}
Let $\Xc$ be a stack, which is locally a linear quotient stack of finite type. Then, the map $\Xc(W)^{\iso} \to \Xc(K)^{\iso}$ has discrete fibres over each point $x \in \Xc(K)$ belonging to the representable locus (i.e., the maximal open substack which is an algebraic space). 
\end{lemma}

\begin{proof}
 By assumption, we can express $\Xc$ as an ascending union $\bigcup_{i\in I} \Xc_i$ where each $\Xc_i$ is a linear quotient stack of finite type. Let $\Xc_i = [U_i/G_i]$ be a presentation as a quotient stack, where we add the additional assumption of $G_i$ being special, i.e., all $G_i$-torsors are Zariski-locally trivial. This property holds for $G_i$ being $GL_r$. It is needed since it guarantees that $H^1(K,G_i)=0$ and thus allows us to define $\Xc_i(K)^{\iso} = U_i(K)/G(K)$.
 
 The claim now follows from the commutative diagram:
 \[
 \xymatrix{
 U_i(W)/G_i(W) \ar[r]^-{\subset}_-{\text{open}} \ar[d]_-{=}^-{\text{def.}} & U_i(K)/G_i(W) \ar[d]^{\alpha_i} \\
  \Xc_i(W)^{\iso} \ar[r] \ar@{..>}[ur] & \Xc_i(K)^{\iso}. 
 }
 \]
 Assume that $x \in \Xc$ is a point in the representable locus, i.e., that $x$ has a Zariski-open neighbourhood representable by an algebraic space. Then, the corresponding $G$-orbit in $U$ is free, i.e., isomorphic to $G_{i,K}=G_i \times_W K$. This shows that the set $\alpha_i^{-1}(x)$ is (non-canonically) homeomorphic to $G_i(K)/G_i(W)$. Since $G_i(W) \subset G_i(K)$ is compact-open, the quotient topology is discrete.

 Thus the intersection of the image of $\Xc_i(W)^{\iso}$ with $\alpha_i^{-1}(x)$ is simultaneously discrete and compact. Thus, $\alpha_i^{-1}(x)$ is finite.

Now, we return to the stack $\Xc$. We infer from the finiteness of the fibre above that the image of $\Xc(W)$ in $\bigcup_{i \in I} \alpha_i^{-1}(x)$ is a discrete space as asserted.
\end{proof}

\begin{corollary}\label{cor:isolated} 
 Suppose that $\Xc$ satisfies the assumptions of Lemma \ref{lemma:discrete}. Assume furthermore that $x$ is an isolated point in $\Xc(K)^{\iso}$ {belonging to the representable locus of $\Xc$}. Then, if $y \in \Xc(W)^{\iso}$ is mapped to the isolated $K$-point $x$, $y$ is an isolated point in $\Xc(W)^{\iso}$. { Vice versa, if $y$ is isolated in $\Xc(W)^{\iso}$ then $x$ is isolated in $\Xc(K)^{\iso}$.}
\end{corollary}
\begin{proof}
 Assume by contradiction that $y$ is not isolated in $\Xc(W)^{\iso}$. Then, there exists a non-discrete neighbourhood $V$ of $y$ which must be entirely mapped to $x$ by the natural continuous map $\Xc(W)^{\iso} \to \Xc(K)^{\iso}$. This contradicts Lemma \ref{lemma:discrete}.

{ For the other direction, we use that the map $\Xc(W)^{\iso} \to \Xc(K)^{\iso}$ is open. This follows from \cite[Proposition 2.9(e)]{Cesnavicius}. Note that the required property of \emph{\'etale-openness} of $W$ and $K$ is shown in (2) and (3) in \cite[Subsection~2.8]{Cesnavicius}.}
\end{proof}

{ The representability assumption on $x$ could be removed, as was pointed out to us by Will Sawin. It is included above for the sake of simplifying the argument.}

\section{Rigidity}

\subsection{Finiteness}
A crucial property of rigid flat connections is that they are of finite number, once the numerical invariants ($r$ and $n$) are fixed. 

\begin{lemma}
There is at most a finite number of isolated points in $\Mc_{dR}(K)$.    
\end{lemma}
\begin{proof}
A point $x=[(E,\nabla)]$, which is $p$-adically isolated is also isolated with respect to the Zariski topology. We choose an embedding $W \subset \mathbb{C}$ and consider the base change $(\Mc_{dR})_{\mathbb{C}}$. The Zariski-isolated point $x \in \Mc_{dR}(\Cb)$ is also isolated with respect to the strong topology. It suffices therefore to show that there is at most a finite number of isolated points in the complex-analytic space. According to the Riemann-Hilbert correspondence, $x$ corresponds to an isolated point in the Betti moduli stack $\Mc_B$, parametrising representations of the fundamental group of $X_{\Cb}$ with quasi-unipotent monodromies along $D$, with eigenvalues contained in $e^{\frac{2\pi i}{n}\Zb}$. The Betti moduli stack admits an affine atlas, and thus contains at most finitely many isolated points.
\end{proof}

We remark that it would be possible to avoid the Riemann-Hilbert correspondence in the proof above, and to replace it by a completely algebraic argument. We briefly sketch the argument. The first step is to show that a flat log-connection $(E,\nabla)$ on $(\Xc_{D,n})_K$ with nilpotent residues along the boundary divisor has vanishing orbifold Chern classes. This can be shown as in~\cite[Proposition~B.1]{EV86}. We remark that \emph{loc. cit.} verifies this for log-connections on a smooth proper variety. A similar argument applies to the case of root stacks, \emph{mutatis mutandis}.

One can then apply Langer's boundedness result \cite{Lan14} (established purely algebraically) to infer finiteness of the number of isolated points of $\Mc_{dR}(K)$.

\subsection{$F$-isocrystals}
 We prove the first part of Main Theorem~\ref{thm:main}.
As seen in Corollary \ref{cor:isolated}, a $W$-point of the {stable locus} of $\Mc_{dR}$ is isolated if and only if the induced $K$-point is isolated {(since stability implies that all automorphisms are scalar)}. Furthermore, we showed that $F^*$ preserves isolated $W$-points in Corollary \ref{cor:F-isolated}. By injectivity of $F^*$ (see Lemma \ref{lemma:injective}), we obtain a permutation of the finite set of isolated points {within the representable locus}. Therefore, each rigid flat connection gives rise to a finite $F^*$-orbit, that is, an $F$-isocrystal.

If $n=1$, then we are given a Frobenius log-crystal on a proper variety $\bar{X}_k$, and by \cite[Theorem 6.4.5]{Kedlaya}, it is overconvergent. In the general case, we are given a Frobenius log-crystal on the root stack $\Xc_{D,n}$ and we cannot directly apply \emph{loc. cit.}. To resolve this problem we may choose appropriate \'etale coordinates on $\bar{X}$, and assume without loss of generality that we are given a Frobenius log-isocrystal $(E,\nabla)$ on the open subset $X$ of the order $n$ root stack $\Yc$ of $\Ab_k^d$ associated to the divisor $\{z_1\cdots z_m=0\}$. We consider the ramified covering $g\colon \Ab_k^d \to \Ab_k^d$ given by $z_i \mapsto z_i^n$ for $0\leq i \leq m$. It lifts to an \'etale map $\tilde{g}$ to the root stack $\Ab_k^d\to \Yc$.

According to \cite[Theorem 6.4.5]{Kedlaya}, the pullback $\tilde{g}^*(E,\nabla)$ restricts therefore to an overconvergent $F$-isocrystal on the open subset $\Ab^d_k \setminus \{z_1\cdots z_m=0\}$. Using that the map $g|_{\Ab^d_k \setminus \{z_1\cdots z_m=0\}}$ is finite and surjective, we obtain that it has the effective descent property for overconvergent isocrystals (\cite[Theorem 5.1]{Laz22}). This shows that $(E,\nabla)|_X$ is overconvergent, since we just showed that $g^*(E,\nabla)|_X$ is overconvergent.

\subsection{Higgs-de Rham  flows for rigid local systems}

We now have the tools to address the question when a rigid flat connection on $\bar{\Xc}_{D,n}/W$ gives rise to an $f$-periodic HdR-flow, that is, we prove the second part of Main Theorem~\ref{thm:main}.

\begin{corollary}
Let $r \in \Nb$ be a fixed rank and assume that every rigid flat connection $(E,\nabla)$ in  $\Mc_{dR}(K)$ has a stable $W$-model. Assume furthermore that every such $W$-model is endowed with a Griffiths-transverse filtration defined over $W$  which is locally split such that the Higgs bundle  $(gr^{\rm Fil} E, gr^{\Fil} \nabla)$ is stable.  Then, every rigid flat connection is periodic.    
\end{corollary}
\begin{proof}
Frobenius pullback preserves isolated points of the moduli space by Theorem \ref{thm:open}. The assertion now follows from Corollary \ref{cor:W-model} . 
\end{proof}

\end{document}